\documentclass[12pt,a4paper]{article}
\usepackage[english]{babel}

\usepackage{pstricks}
\usepackage{pst-node}
\usepackage{amsmath,boxedminipage}
\usepackage{amssymb}
\usepackage{graphicx}
\usepackage{enumerate}
\usepackage{color}
\usepackage{cite}
\usepackage[text={15cm,24cm}]{geometry}
\usepackage[normalem]{ulem}
\usepackage[ansinew]{inputenc}
\usepackage{ulem}

\usepackage{subcaption}
\usepackage{float}

\newenvironment{proof}{{\bf Proof}:\ }%
   {~\ \hfill $\Box$\vspace{0,5cm}}

    {~\ \hfill$\Box$\vspace{0,5cm}}
\newenvironment{ack}{\vskip5mm{\bf Acknowledgements:}}%

\newtheorem{theorem}{Theorem}[section]
\newtheorem{prop}[theorem]{Property}
\newtheorem{rmk}[theorem]{Remark}

\newtheorem{coro}[theorem]{Corollary}

\newtheorem{fact}[theorem]{Fact}

\graphicspath{{.}{graphics/}}
\newrgbcolor{lightlightlightgray}{0.9 0.9 0.9}

\numberwithin{equation}{section}

\begin{document}
\title{The complexity of the Perfect Matching-Cut problem}

\author{V. Bouquet\footnotemark[1],
C.\ Picouleau\footnotemark[1]}
\date{\today}

\def\thefootnote{\fnsymbol{footnote}}

\footnotetext[1]{ \noindent
Conservatoire National des Arts et M\'etiers, CEDRIC laboratory, Paris (France). Email: {\tt
valentin.bouquet@cnam.fr,christophe.picouleau@cnam.fr}}

\maketitle

\begin{abstract}
   Perfect Matching-Cut is the problem of deciding whether a graph has a perfect matching that contains an edge-cut. We show that this problem is NP-complete for planar graphs with maximum degree four, for planar graphs with girth five, for bipartite five-regular graphs, for  graphs of diameter three and for bipartite graphs of diameter four. We show that there exist polynomial time algorithms for the following  classes of graphs: claw-free, $P_5$-free, diameter two, bipartite with diameter three and graphs with bounded tree-width.

 \vspace{0.2cm}
\end{abstract}
\section{Introduction}\label{intro}

The Matching-Cut problem consists of finding a matching that is an edge-cut. This problem has been extensively studied. Chv\'atal \cite{Chvatal} proved that the problem is NP-complete for graphs with maximum degree four and polynomially solvable for graphs with maximum degree three.
Bonsma \cite{Bonsma} showed that this problem remains NP-complete for planar graphs with maximum degree four and for planar graphs of girth five. Bonsma \cite{BonsmaFarley} et al. showed that the planar graphs with girth at least six have a matching-cut. They also gave  polynomial algorithms for some subclasses of graphs including claw-free graphs, cographs and graphs with fixed bounded tree-width or clique-width.
Le and Randerath \cite{LeRanderath} showed that the Matching-Cut problem is NP-complete for bipartite graphs in which the vertices of one side of the bipartition have degree three and the other vertices have degree four.

We address the Perfect Matching-Cut problem where the matching involved in the matching-cut is contained in a perfect matching.  Figure \ref{matchcut} shows on the left a graph having a perfect matching and a matching-cut but no perfect matching-cut (the matching-cut is outlined by the bold edges), on the right a graph having a perfect matching-cut.
To our knowledge, the only reference to this problem is from Diwan \cite{Diwan} which he called \textit{Disconnected $2$-Factors}. The graphs he studied are cubic, so finding a disconnected $2$-factor is the same as finding a perfect matching-cut. It is shown that all planar cubic bridgeless graphs have a disconnected $2$-factors except $K_4$.  Note that this is an enhancement, for the planar graphs, of the famous  Petersen's theorem that every cubic bridgeless graph has a perfect matching.

A variant of the Perfect Matching-Cut problem is studied in \cite{Heggernes} where it is shown that it is NP-complete to recognize a graph having a partition of the vertices into $(X,Y)$ such that the edges between $X$ and $Y$ forms a perfect matching. In Figure \ref{matchcut}, the graph on the right has a perfect matching-cut that is outlined by the bold edges, but it has no partition $(X,Y)$ of the vertices such that the edges between $X$ and $Y$ is a perfect matching.

\begin{figure}[htbp]
   \begin{center}
      \includegraphics[width=15cm, height=3cm, keepaspectratio=true]{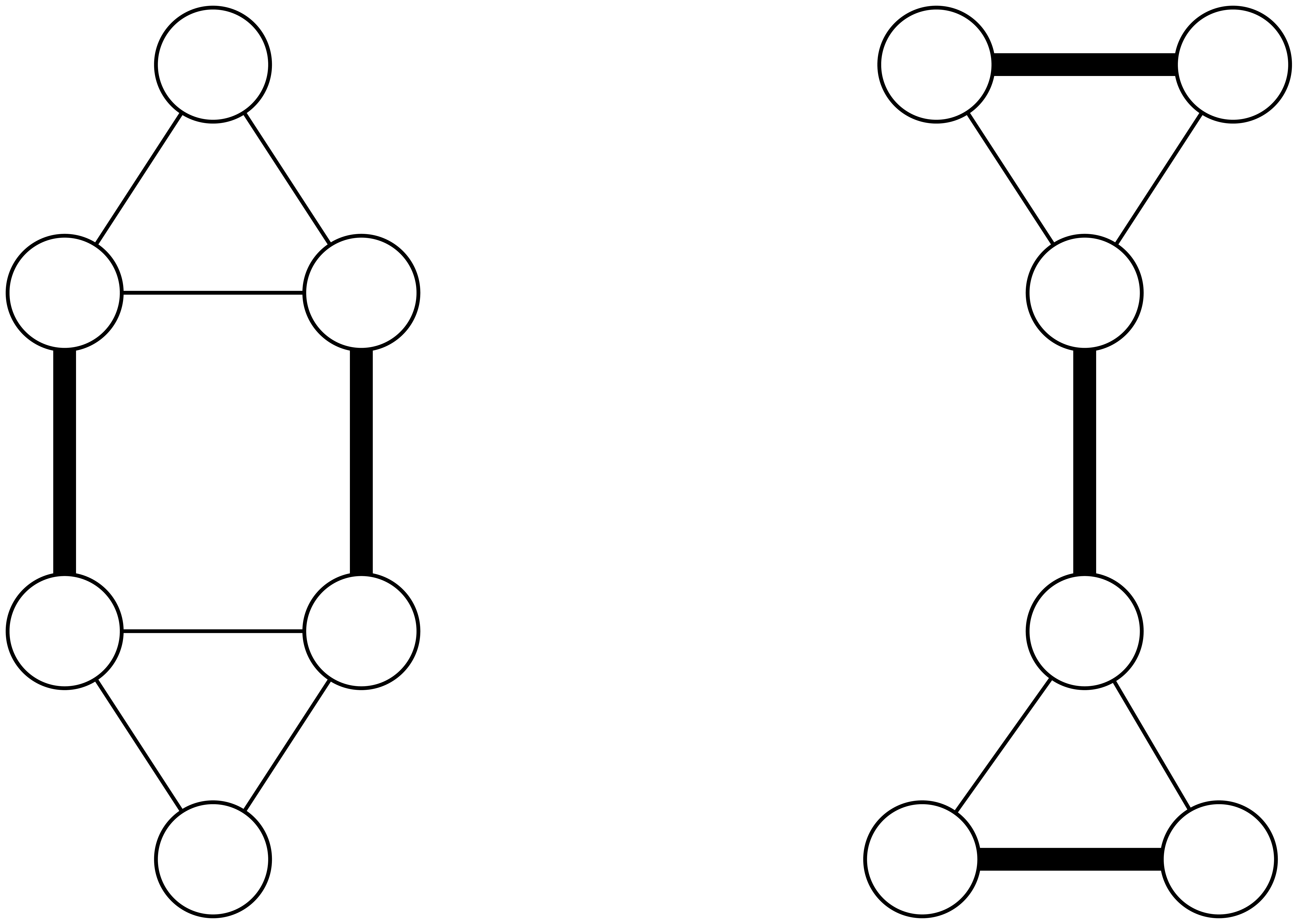}
   \end{center}
   \caption{On left a matching-cut, on right a perfect matching-cut.}
   \label{matchcut}
\end{figure}

In the same spirit of previous work on the Matching-Cut problem, we explore the complexity of the Perfect Matching-Cut problem. After some observations in Section \ref{pref}, we prove the following. In Section \ref{bipdiam}, the Perfect Matching-Cut problem is shown to be NP-complete for the $5$-regular bipartite graphs, for the graphs with  diameter three and for the bipartite graphs with diameter $d$, for any fixed $d\ge 4$,  and the problem is shown to be polynomial for the bipartite graphs with diameter three and for the graphs of diameter two. In Section \ref{planar}, the Perfect Matching-Cut problem is shown to be NP-complete for planar graphs with degrees three or four, and for planar graphs with girth five. In Sections \ref{P5free}, \ref{claw} and \ref{width}, the Perfect Matching-Cut problem is shown to be polynomial for the classes of $P_5$-free graphs (which includes cographs, split graphs, co-bipartite graphs), claw-free graphs and graphs with bounded tree-width. The section \ref{conc} is devoted to a conclusion and some open questions.

\section{Notations and preliminaries}\label{pref}
All graphs considered in this paper are finite, undirected, simple and loopless. For a graph $G=(V,E)$, let $\delta(G),\, \Delta(G),$ denote its minimum degree, its maximum degree, respectively.  The degree of a vertex $x$ is denoted by $d_G(x)$ or simply $d(x)$ when the context is unambiguous. A graph is {\it even} when $\vert V\vert$ is even, otherwise it is {\it odd}. A graph $G$ is $k$-regular whenever all its vertices have the same degree $d(x)=k$. For $v\in V$,  $N(v)$ is its neighborhood, $N[v]=N(v)\cup\{v\}$ its closed neighborhood.
For $X\subseteq V$, let $N(X)=\bigcup_{x\in X} N(x)$ and $N[X]=\bigcup_{x\in X} N[x]$. The length of a shortest path between two vertices $u,v$ of $G$ is $dist_G(u,v)$. Let $diam(G)$ denote the diameter of $G$, that is the maximum length of a shortest path (as usual when $G$ is not connected then $diam(G)=\infty$).  Let $g(G)$ denote the girth of $G$, that is the minimum length of a cycle. The graph $H$ is a subgraph of $G$, $H\subset G$, if its vertex set $V(H)\subset V$ and its edge set $E(H)\subset E$ is such that every $uv\in E(H)$ is an edge of $E$. For  $M\subseteq E$ we write $G-M$ the partial graph $G=(V,E\setminus M)$. A set $M\subseteq E$ is a {\it matching} when no edges of $M$ share a vertex. A matching $M$ is {\it perfect} when every vertex is incident to one edge of $M$.

Given $X\subseteq V$, the subgraph of $G$ induced by $X$, that is the graph with the vertex set $X$ and all the edges $xy\in E$ with $x\in X,y\in X$, is denoted by $G[X]$.  We write $G-v=G[V\setminus \{v\}]$ and for a subset $V'\subseteq V$ we write $G-V'=G[V\setminus V']$. The {\it contraction} of an edge $uv\in E$ removes the vertices $u$ and $v$ from $G$, and replaces them by a new vertex denoted $uv$ made adjacent to precisely those vertices that were adjacent to $u$ or $v$ in $G$ (neither introducing self-loops nor multiple edges). This operation is denoted $G/uv$.

The complete graph with $n$ vertices is $K_n$, also called the {\it clique} on $n$ vertices. The clique $C_3=K_3$ is a {\it triangle}. Let $K_{m,n}$ be the {\it bipartite clique} with $m$ vertices in one side of the partition and $n$ vertices in the opposite side. The {\it claw} is denoted $K_{1,3}$. Let $C_n$ be the chordless cycle on $n$ vertices. For a fixed graph $H$, the graph $G$ is {\it $H$-free} if $G$ has no induced subgraph isomorphic to $H$. For two vertex disjoint induced subgraphs $G[A],\, G[B]$ of $G$, $G[A]$ is {\it complete} to $G[B]$ if $ab\in E$ for any $a\in A,\, b\in B$. Other notations or definitions  of graphs that are not given here can be found in \cite{Bondy}.\\

A {\it cut} in $G$ is a partition $V=X\cup (V\setminus X)$ with $X,V\setminus X\ne \emptyset$. The set of all edges in $G$ having an endvertex in $X$ and the other endvertex in $V\setminus X$, also written $E(X,V\setminus X)$, is called the {\it edge-cut} of the cut. When the context is unambiguous we can use cut for edge-cut.
A \textit{bridge} is an edge-cut with exactly one edge. Note that when $G$ is connected every edge-cut $C=E(X,V\setminus X)$ is such that $G-C$ is disconnected; we say that $C$ disconnects $G$. A {\it matching-cut} is an edge-cut that is a matching. A {\it perfect matching-cut} is a perfect matching that contains a matching-cut. When a connected graph or subgraph cannot be disconnected by a matching it is called {\it immune}, when it cannot be disconnected by a perfect matching it is called {\it perfectly immune}.\\

The decision problem associated with the perfect matching-cut is defined as:
\begin{center}
\begin{boxedminipage}{.99\textwidth}
\textsc{\sc Perfect Matching-cut\ (PMC)} \\[2pt]
\begin{tabular}{ r p{0.8\textwidth}}
\textit{~~~~Instance:} &a connected graph $G=(V,E)$.\\
\textit{Question:} &is there $M\subset E$ such that $M$ is a perfect matching and $G-M$ is not connected ?
\end{tabular}
\end{boxedminipage}
\end{center}

Note that $PMC$  is clearly in NP.\\

Here we give some easy observations.
\begin{fact}\label{discclique}
$K_3$ and $K_{2,3}$ are immune.
\end{fact}
\begin{fact}\label{clique}
Assume that $G$ has a perfect matching-cut $M$ with $E(X,V\setminus X)\subset M$ and let $H$ be an immune subgraph of G. Then either $V(H)\subset X$ or $V(H)\subset V\setminus X$.
\end{fact}
\begin{fact}\label{2neighX}
Assume that $G$ has a perfect matching-cut $M$ with $E(X,V\setminus X)\subset M$.
 If $v$ has two neighbors in $X$ then $v\in X$.
\end{fact}
\begin{proof}
Let $M$ be a perfect matching. There is a unique vertex $u$ such that $uv\in M$. So in $G-M$, $v$ has a neighbor in $X$. Hence $v\in X$.
\end{proof}

\begin{fact}\label{bridge}
If $M$ is a perfect matching that contains a bridge then $M$ is a perfect matching-cut.
\end{fact}

\begin{fact}\label{leaf}
When a graph $G$ with $\delta(G)=1$ has a perfect matching then $G$ has a perfect matching-cut.
\end{fact}

\section{Bipartite graphs and bounded diameter}\label{bipdiam}
We give our results for bipartite and regular graphs, and  for graphs with fixed diameter. First we show NP-complete results, second polynomial problems.
\subsection{NP-complete classes}
\begin{theorem}\label{NPcbip5reg}
 Deciding if a connected bipartite $5$-regular graph has a perfect matching-cut is NP-complete.\end{theorem}
\begin{proof}
By Le and Randerath \cite{LeRanderath} we know the following: deciding if a bipartite graph $G=(V=S\cup T,E)$, such that all the vertices of $T$ have degree three and all the vertices of $S$ have degree four, has a matching-cut is NP-complete.

From $G$ we build (in polynomial time) a connected  graph $G'=(V',E')$ as follows. We take two copies $G_1=(V_1=S_1\cup T_1,E_1)$, $G_2=(V_2=S_2\cup T_2,E_2)$ of $G$. For each $v\in T$, let $v_1\in T_1$, $v_2\in T_2$, be its two copies in $G_1$, $G_2$, respectively. We link $v_1$ to $v_2$ with the graph depicted in the figure \ref{bip5reg}.
For each $v\in S$, let $v_1\in S_1$, $v_2\in S_2$ be its two copies in $G_1,G_2$, respectively. We add the edge $v_1v_2$. Clearly $G'$ is bipartite and  $5$-regular. One can observe that the graph depicted in the figure \ref{bip5reg} is immune.

From $M\subset E$ a matching-cut of $G$ corresponds $M'\subset E'$ as follows. For each $uv\in M,\, u\in S,\,v\in T$, we take $u_1v_1,u_2v_2$  and the edges  represented by the bold lines on the left of the figure \ref{bip5reg}. For each vertex $w$ which is not covered by $M$: when $w\in T$ we take the  edges represented by the bold lines on the right of the figure \ref{bip5reg}; when $w\in S$ we take the  edge $w_1w_2$. Hence $M'$ is a perfect matching of $G'$.
There are two vertices $u,v$ in $G-M$ with no path between them, so there is no path linking $u_i$ with $v_j$, $i,j\in\{1,2\}$ in $G'-M'$. So $M'$ is a perfect matching-cut of $G'$.

Reciprocally, let $M'$ be a perfect matching-cut of $G'$.  The graph shown by the figure \ref{bip5reg} cannot be disconnected by $M'$. Moreover the edges $v_1v_2,\,v\in S,$ cannot disconnect $G'$. So we have that $G'[V_i],i\in\{1,2\}$ is disconnected. Let $M_1=M'\cap E_1$, for instance. To $M_1$ corresponds a matching-cut of $G$.
\end{proof}

\begin{figure}[htbp]
\begin{center}
\includegraphics[width=15cm, height=5cm, keepaspectratio=true]{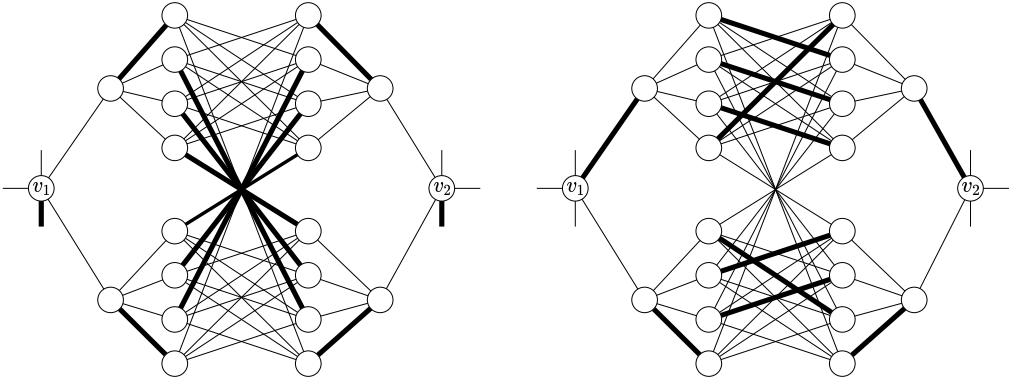}
\end{center}
\caption{On the left $v$ is covered by $M$, on the right $v$ is not covered by $M$.}
\label{bip5reg}
\end{figure}

\begin{theorem}
   Deciding if a graph of diameter $3$ has a perfect matching-cut is NP-complete.
\end{theorem}
\begin{proof}
By Le and Le \cite{lele}, we know that deciding if a graph with diameter three has a matching-cut is NP-complete. Their proof inspired ours.

From a graph $G=(V,E)$, we build (in polynomial time) a connected graph $G'=(V',E')$ as follows.
We take two copies $G_1=(V_1,E_1)$, $G_2=(V_2,E_2)$ of $G$. We denote  $V=\{v_1,v_2,\ldots,v_n\},\, V_1=\{v_1^1,v_2^1,\ldots,v_n^1\},\, V_2=\{v_1^2,v_2^2,\ldots,v_n^2\}$. We describe how to build $G_1$, the construction for $G_2$ is the same. For each $1\le i\le n$, let $K_{v_i^1}$ be a complete graph of $n$ vertices. We add edges between $v_i^1$ and $K_{v_i^1}$, so that $K_{v_i^1}\cup \{v_i^1\}$ is a clique of size $n+1$. We say that $K_{v_i^1}$ is the associated clique of $v_i^1$. Let $K^1=\bigcup_{i=1,\ldots,n} V(K_{v_i^1})$ be the set of vertices of the associated cliques of $V_1$. For each pair of cliques $K_{v_i^1},\, K_{v_j^1},\, i<j$, we add exactly one edge $q_iq_j$, $q_i\in K_{v_i^1}$, $q_j\in K_{v_j^1}$, in such a way that $q_i,q_j$ has exactly one neighbor in $K^1\setminus V(K_{v_i^1})$, $K^1\setminus V(K_{v_j^1})$, respectively. Hence for each clique $K_{v_i^1}$, every vertex of $K_{v_i^1}$ --- except one that has no neighbor in $K^1\setminus V(K_{v_i^1})$ --- has exactly one neighbor in $K^1\setminus V(K_{v_i^1})$. Then, for each pair $v_i^1\in V_1,\, v_i^2\in V_2$ we make $K_{v_i^1}$ complete to $K_{v_i^2}$, $v_i^1$ complete to $K_{v_i^2}$ and $v_i^2$ complete to $K_{v_i^1}$.  Hence $K_{v_i^1,v_i^2}=K_{v_i^1}\cup K_{v_i^2}$ is a clique with $2n$ vertices. Also, note that $K_{v_i^1,v_i^2}\cup \{v_i^1\}$ and $K_{v_i^1,v_i^2}\cup \{v_i^2\}$ are two cliques with $2n+1$ vertices, and therefore $K_{v_i^1,v_i^2}\cup \{v_i^1,v_i^2\}$ is immune. At this point $G'$ is constructed. \\

We show that $diam(G')=3$. First, note that for every pair $q\in K_{v_i^1}$, $q'\in K_{v_j^1}$, we have $dist_{G'}(q,q')\leq 3$, and for every pair $v\in V_1\cup V_2,\, q\in K_u$, with $u\in V_1\cup V_2$, we have $dist_{G'}(v,q)\leq 3$.
Let $v_i^1,v_j^1,\, i\neq j,$ be two non adjacent vertices of $G'$. From our construction, there exists $qq'\in E(G')$ where $q\in K_{v_j^1}$  and $q'\in K_{v_i^1}$. Recall that $v_i^1$ and $v_j^1$ are both complete to their associated clique. Hence $v_i^1-q-q'-v_j^1$ is a path and therefore $dist_{G'}(v_i^1,v_j^1)\leq 3$. The same holds for two distinct vertices of $V_2$. Now let a pair $v_i^1,v_j^2$. If $i=j$ then there exists a path $v_i^1-q-v_i^2$, where $q\in K_{v_i^1,v_i^2}$. Now let $i\neq j$. There exists $qq'\in E(G')$ where $q$ is a vertex of $K_{v_i^1}$ and $q'$ is a vertex of $K_{v_j^2}$. Hence $v_i^1-q-q'-v_j^2$ is a path and therefore $dist_{G'}(v_i^1,v_j^2)\leq 3$. So diam$(G')=3$. \\

From $M\subset E$ a matching-cut of $G$ corresponds $M'\subset E'$ as follows. Let $(X,Y)$ be the partition of $G$ induced by the matching-cut $M$. For each pair $x\in X$, $y\in Y$ let $x^1,y^1\in V_1$, and $x^2,y^2\in V^2$ be their corresponding vertices in $G'$. Then we take $q_{x^1}q_{y^1}$ and $q_{x^2}q_{y^2}$ in $M'$, where $q_{x^1}\in K_{x^1}$, $q_{x^2}\in K_{x^2}$, $q_{y^1}\in K_{y^1}$ and $q_{y^2}\in K_{y2}$. We show that $G'[K^1\cup K^2]-M'$ is disconnected. Let $K^X$ be the set of cliques associated with a vertex $v\in X$. The set $K^Y$ is defined similarly for $Y$. For each pair $K_u\in K^X,K_v\in K^Y$, there is exactly one edge $qq'$ between $K_u$ and $K_v$, and therefore $qq'\in M'$. Hence $G'[K^1\cup K^2]-M'$ is disconnected and $M'$ is a matching.
For each $v_iv_j\in M$ we take $v_i^1v_j^1,v_i^2v_j^2\in M'$. Hence $G_1-M'$ and $G_2-M'$ are both disconnected. Recall that $G'[K^1\cup K^2]-M'$ is also disconnected, and thus $G'-M'$ is disconnected.
So $M'$ is a matching-cut of $G'$. We show how to take the remaining edges so that $M'$ is a perfect matching-cut. Let the vertices that are not covered by $M$. We chose two uncovered vertices $q_1,q_2$ in $K_{v_i^1,v_j^2}$ and  take $v_i^1q_1,v_j^2q_2$ in $M'$. Note that it remains an even number of vertices in $K_{v_i^1,v_j2}$ that are not covered by $M'$, so we can take any matching covering these edges in $M'$. Clearly, $M'$ is a perfect matching-cut.

Reciprocally, let $M'$ be a perfect matching-cut of $G'$. Recall that for each pair $v_i^1,v_i^2$, the clique $K_{v_i^1,v_i^2}\cup \{v_i^1,v_i^2\}$ is immune. So we have that $G'[V_i]$, $i\in\{1,2\}$ is disconnected by $M'$. Let $M_1=M'\cap E_1$, for instance. To $M_1$ corresponds a matching-cut of $G$.
\end{proof}

\begin{figure}[htbp]
   \begin{center}
   \includegraphics[width=15cm, height=5cm, keepaspectratio=true]{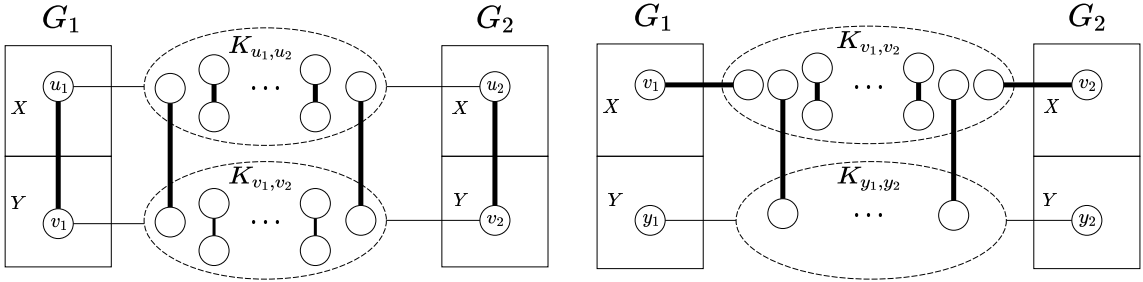}
   \end{center}
   \caption{On the left $u_1,v_1$ are covered by $M$, on the right $v_1,v_2$ are not covered by $M$.}
   \label{3diam}
\end{figure}

\begin{figure}[htbp]
   \begin{center}
   \includegraphics[width=8cm, height=5cm, keepaspectratio=true]{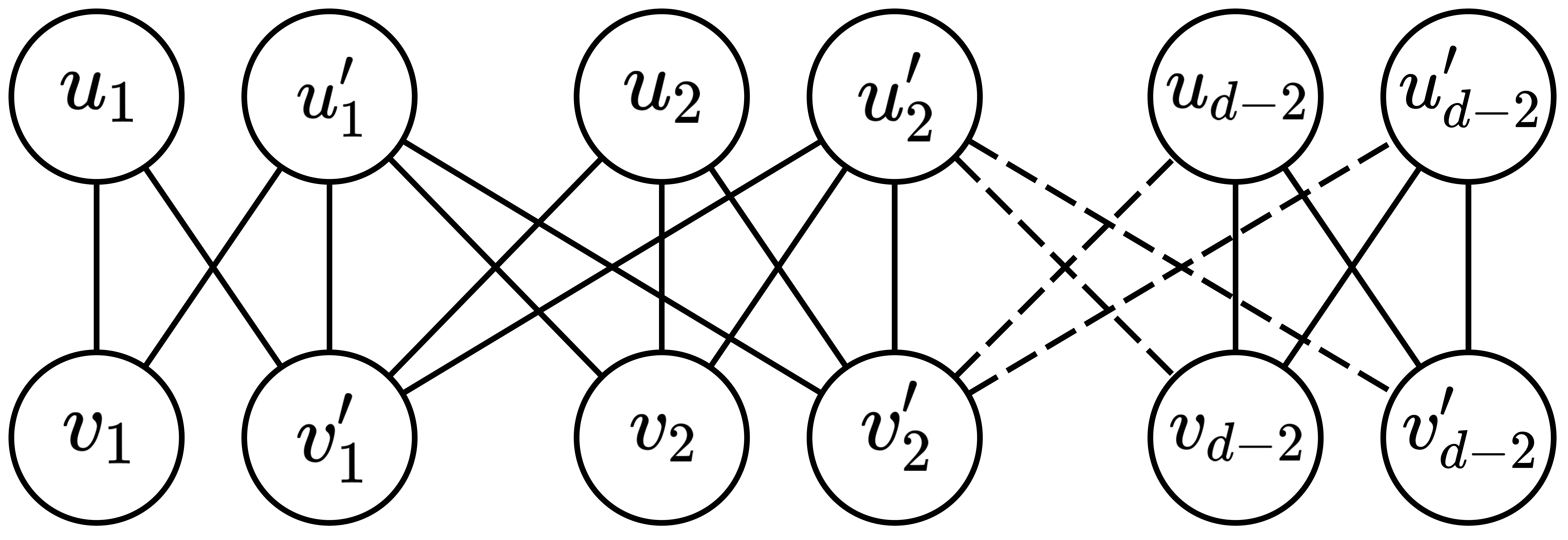}
   \end{center}
   \caption{The bipartite chain composed of $K_{3,3}$}
   \label{kdiam}
\end{figure}

\begin{theorem}\label{NPcbip4diam}
   Deciding if a  bipartite graph $G$ with $diam(G)=4$ has a perfect matching-cut is NP-complete.
\end{theorem}
\begin{proof}
In Le and Le \cite{lele} it is shown that deciding whether a bipartite graph $G$ with $diam(G)=4$ has a matching-cut is NP-complete. Their proof inspired ours.

From a bipartite graph $G=(V=A\cup B,E)$ with $diam(G)=3$, we build (in polynomial time) a connected graph $G'=(V',E')$ as follows.
We take two copies $G_1=(V_1=A_1\cup B_1,E_1)$, $G_2=(V_2=A_2\cup B_2,E_2)$ of $G$. We denote  $V=\{v_1,v_2,\ldots,v_n\},\, V_1=\{v_1^1,v_2^1,\ldots,v_n^1\},\, V_2=\{v_1^2,v_2^2,\ldots,v_n^2\}$. For each $1\le i\le n$, let $Q_{v_i^1}$ be an independent set of $n$ vertices. We add edges between $v_i^1$ and $Q_{v_i^1}$ so that $v_i^1$ is complete to $Q_{v_i^1}$. We say that $Q_{v_i^1}$ is the associated independent set of $v_i^1$. We do the  same  with the vertices of $G_2$. Then for each pair $v_i^1\in V_1,\, v_i^2\in V_2$ we make $Q_{v_i^1}$ complete to $Q_{v_i^2}$.  So $Q_{v_i^1}\cup Q_{v_i^2}$ is isomorphic to $K_{n,n}$. We denote the following: $Q_{v_i^1,v_i^2}=Q_{v_i^1}\cup Q_{v_i^2}\cup \{v_i^1,v_i^2\}$, $Q_{A_k}=\bigcup_{v\in A_k}Q_v$, $Q_{B_k}=\bigcup_{v\in B_k}Q_v$, $Q_k=Q_{A_k}\cup Q_{B_k}$, $k=1,2$, and $Q=Q_{1}\cup Q_{2}$. Note that $Q_{v,v'}$ is immune.

For each pair $v_i\in A,v_j\in A,i\neq j,$ we add exactly one edge between $Q_{v_i^1}$ and $Q_{v_j^2}$ and exactly one edge between $Q_{v_i^2}$ and $Q_{v_j^1}$. We perform this operation in such a way that the edges we added form a matching. We perform in a same manner with the pairs $v_i\in B,v_j\in B,i\neq j$. For each pair $v_i\in A,v_j\in B$, we add exactly one edge between $Q_{v_i^1}$ and $Q_{v_j^1}$ and exactly one edge between $Q_{v_i^2}$ and $Q_{v_j^2}$. We perform this operation in such a way that these edges plus the previous one form a matching. At this point, $G'$ is constructed.
Note that $A_1\cup Q_{B_1}\cup Q_{A_2}\cup B_2$ and $B_1\cup Q_{A_1}\cup Q_{B_2}\cup A_2$ are two disjoint independent sets that form a partition of $V'$ the vertex-set of $G'$. So $G'$ is bipartite. \\

We show that $diam(G')=4$. Since $diam(G)=3$, for every pair $v_i,v_j\in V_k$, $k=1,2$, we have $dist_{G'}(v_i,v_j)\leq 3$. Then for every $v_i\in V_k$, $q\in Q_k$, $i=1,2$, we have $dist_{G'}(v_i,q)\leq 4$. Note that since $diam(G)=3$, for every pair $v_i\in A,\, v_j\in A,\, i\ne j$ or $v_i\in B,\, v_j\in B,\, i\ne j$ we have $dist_{G}(v_i,v_j)=2$. Then for every pair $q,q'\in Q_{A_1}$ we have $dist_{G'}(q,q')\le 3$. The same holds for every pair in $Q_{A_2}$ or $Q_{B_1}$ or $Q_{B_2}$. Let $q_i\in Q_{v_i^1},\, q_j\in Q_{v_j^2}$. When $i=j$ then $dist_{G'}(q_i,q_j)=1$. Now, let $i\ne j$. First, $v_i,v_j\in A$. There exists a path $P=q_i-q_2-q_1-q_j$ where $q_2q_1$ is the edge between $Q_{v_i^2}$ and $Q_{v_j^1}$, so $dist_{G'}(q_i,q_j)\le 3$. Second, $v_i\in A,\, v_j\in B$. There exists a path $P=q_i-v_i-q_a-q_b-q_j$ where $q_aq_b$  is the edge between $Q_{v_i^1}$ and $Q_{v_j^1}$, so $dist_{G'}(q,q')\le 4$. Now, let $q_i\in Q_{v_i^1},\, q_j\in Q_{v_j^1}$ with $v_i\in A,\, v_j\in B$. There exists a path $P=q_i-q_a-q_b-q_j$ where $q_aq_b$  is the edge between $Q_{v_i^2}$ and $Q_{v_j^2}$, so $dist_{G'}(q_i,q_j)\le 3$. Now, let $q_i\in Q_{v_i^1},\, v_j^2\in V_2$. First, $v_i,v_j\in A$. There exists a path $P=q_i-v_i-q_i'-q_j-v_j$ where $q_i'q_j$ is the edge between $Q_{v_i^1}$ and $Q_{v_j^2}$, so $dist_{G'}(q_i,v_j)\le 4$. Second, $v_i\in A,\, v_j\in B$. There exists a path $P=q_i-q_a-q_j-v_j$ where $q_aq_j$  is the edge between $Q_{v_i^2}$ and $Q_{v_j^2}$, so $dist_{G'}(q_i,v_j)\le 3$. All the remaining cases are symmetric. So $diam(G')=4$. \\

From $M\subset E$ a matching-cut of $G$ corresponds $M'\subset E'$ as follows. Let $(X,Y)$ be a partition of $G$ induced by the matching-cut $M$. Let $x\in X$, $y\in Y$.
If $x\in A$ and $y\in B$ then we take $q_{x_1}q_{y_1}, q_{x_2}q_{y_2}$ in $M'$. This operation is depicted on the left of Figure \ref{bip4diam1}. Otherwise $x,y\in A$, resp. $x,y\in B$, and we take $q_{x_1}q_{y_2},q_{x_2}q_{y_1}$ in $M'$. See figure \ref{bip4diam1} on the right. For each $uv\in M$ we take $u_1v_1,u_2v_2$ in $M'$. At this point we clearly have a matching-cut of $G'$. We show how to take the remaining edges so that $M'$ is a perfect matching-cut.

Let $v_i\in V$ be a vertex that is not covered by $M$. We chose two uncovered vertices $q_1,q_2$ in $Q_{v_i^1,v_i^2}$ and  take $v_i^1q_1,v_i^2q_2$ in $M'$. Now, for each $Q_{v_i^1,v_i^2}$ it remains the same number of uncovered vertices in $Q_{v_i^1}$ as in $Q_{v_i^2}$. Since $Q_{v_i^1,v_i^2}$  is a complete bipartite graph we can take any matching with the remaining uncovered vertices in $M'$.
This operation is depicted by the figure \ref{bip4diam2}. Clearly, $M'$ is a perfect matching-cut.

Reciprocally, let $M'$ be a perfect matching-cut of $G'$. Recall that for each pair $v_i^1\in V_1$, $v_i^2\in V_2$,  the subgraph $Q_{v_i^1,v_i^2}$ is immune. So $G'[V_i],i\in\{1,2\}$ is disconnected. Let $M_1=M'\cap E_1$, for instance. To $M_1$ corresponds a matching-cut of $G$.
\end{proof}

\begin{figure}[htbp]
   \begin{center}
   \includegraphics[width=15cm, height=5cm, keepaspectratio=true]{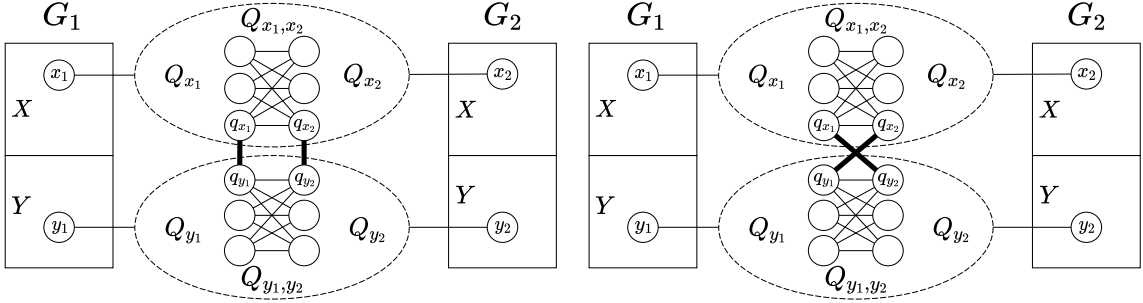}
   \end{center}
   \caption{On the left, $x_1\in A_1, x_2\in A_2, y_1\in B_1, y_2\in B_2$ and $q_{x_1}q_{y_1}, q_{x_2}q_{y_2}$ are in the perfect matching-cut of $G'$. On the right, $x_1, y_1\in A_1, x_2,y_2\in A_2$ and $q_{x_1}q_{y_2}, q_{x_2}q_{y_1}$ are in the perfect matching-cut of $G'$.}
   \label{bip4diam1}
\end{figure}

\begin{figure}[htbp]
   \begin{center}
   \includegraphics[width=15cm, height=5cm, keepaspectratio=true]{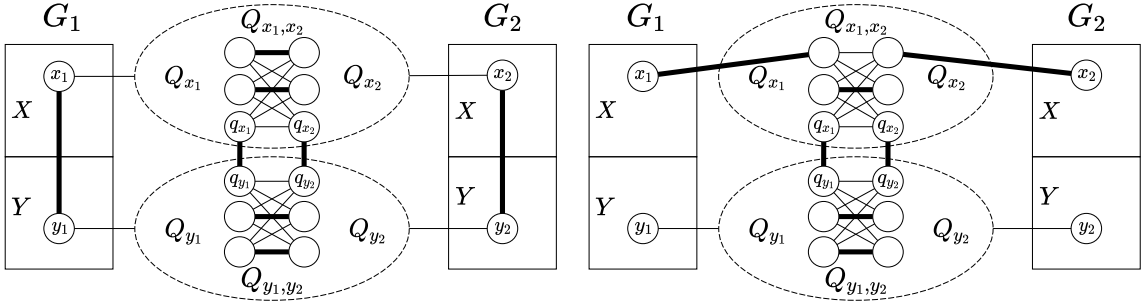}
   \end{center}
   \caption{On the left $x_1y_1$, $x_2y_2$ are covered by the matching-cut of $G$. On the right $x_1$ and $x_2$ are not covered by the matching-cut of $G$.}
   \label{bip4diam2}
\end{figure}

Using the preceding arguments but adding the bipartite graph $H$ of Figure \ref{kdiam} such that $u_1,u_1'$ are complete to $Q_{w^1}$ and $v_1,v_1'$ are complete to $Q_{w^2}$ where $w^1,w^2$ correspond to the two copies of  a vertex $w$ of $G$, we obtain  a bipartite graph $G'$ with $diam(G')=d,d\geq 3$. Clearly, $H\cup Q_{w^1}\cup Q_{w^2}$ is immune and $H$ has a perfect matching $u_1v_1, u_1'v_1',...,u_{d-2}v_{d-2},u_{d-2}'v_{d-2}'$. Hence the following theorem follows.

\begin{theorem}\label{NPckdiam}
   For every fixed $d\ge 4$, deciding if a bipartite graph $G$ with $diam(G)=d$ has a perfect matching-cut is NP-complete.
\end{theorem}

\subsection{Polynomial classes}
Whereas the Perfect Matching-Cut problem is NP-complete for bipartite graphs with diameter four, here we prove that the Perfect Matching-Cut is polynomial  for the bipartite graphs with diameter three. Afterward, we prove that the Perfect Matching-Cut problem is polynomial  for the graphs with diameter two.\\

For our proof  we use the following characterization given in \cite{lele} (see Fact 2).
\begin{fact}\label{factdiam3bip}
Let $G=(V_1\cup V_2, E)$ be a bipartite graph.  Then $diam(G)\le 3$ if and only if $N(a)\cap N(b)\ne\emptyset$ for any pair of distinct vertices $a,b$ in the same side of the bipartition.
\end{fact}

We also use the following observations. Let $G=(V_1\cup V_2,E)$ be a bipartite graph, and $M$ be a perfect matching-cut with $(X,Y)$ its corresponding partition.

\begin{fact}\label{neighmat}
   Let $uv\in M$ with $u\in X, v\in Y$. Then $N(u)\setminus \{v\} \subseteq X$ and $N(v)\setminus \{u\}\subseteq Y$.
\end{fact}

\begin{fact}\label{twoneigh}
   If a vertex $v$ has two neighbors in $X$, respectively in $Y$, then $v\in X$, respectively $v\in Y$.
\end{fact}

\begin{fact}\label{neighxy}
   If a vertex $v$ has a neighbor in $x\in X$ and another neighbor in $y\in Y$, such that $xy\in M$, then $M$ cannot be a perfect matching-cut.
\end{fact}

We are ready to show the following.
\begin{theorem}\label{bipdiam3}
  Let $G=(V_1\cup V_2,E)$ be a bipartite graph with $diam(G)\le 3$. Deciding if $G$ has a perfect matching-cut can be done in polynomial time.
 \end{theorem}
\begin{proof}
We assume that $G$ has a perfect matching, thus $\vert V_1\vert=\vert V_2\vert$. If $uv$ is a bridge and there exists $M$ a perfect matching-cut such that $uv\in M$ then $G$ has a perfect matching-cut. Now we can assume that $\delta(G)\ge 2$ and if $M$ is a perfect matching-cut of $G$ then $\vert E(X,Y)\vert \ge 2$.

We try to build $M$ a perfect matching-cut with a partition $(X,Y)$. We guess two edges $ab,cd$ with $a,d\in V_1$, $b,c\in V_2$ such that  $\{ab,cd\}\subseteq M$. When a perfect matching-cut $M$ is found the algorithm stops, otherwise we try another pair of edges.

We show that when $G$ has a perfect matching-cut, then there exists $M$ a perfect matching-cut such that $\{a,c\}\subseteq X,\, \{b,d\}\subseteq Y$. Let $a\in X,\, b\in Y$. For contradiction we assume that $V_2\subseteq Y$. Since $d(a)\ge 2$, by Fact \ref{twoneigh} we have $a\in Y$, a contradiction.\\

Hence we put $a,c\in X,\,b,d\in Y$. Note that there are $\mathcal{O}(\vert V\vert ^2)$ such combinations.\\

Note that $G$ being bipartite we have $N(a)\cap N(b)=N(a)\cap N(c)=N(b)\cap N(d)=N(c)\cap N(d)=\emptyset$. Moreover, by Fact \ref{neighxy} if there exists $v\in N(a)\cap N(d)$ or $v\in N(b)\cap N(c)$ then $M$ cannot exist. Hence we have $N(a)\cap N(d)=N(b)\cap N(c)=\emptyset$.\\

We define the following sets of vertices:

\begin{itemize}
   \item $A=N(a)\setminus\{b,c\},\, B=N(b)\setminus\{a,d\},\, C=N(c)\setminus\{a,d\},\, D=N(d)\setminus\{b,c\}$;
   \item $S=\{v\in V_1 \mid v\not\in N(b)\cup N(c),\, N(v)\cap A\ne\emptyset,\, N(v)\cap D\ne\emptyset\}$;
   \item $T=\{v\in V_2 \mid v\not\in N(a)\cup N(d),\, N(v)\cap B\ne\emptyset,\, N(v)\cap C\ne\emptyset\}$.
\end{itemize}

We show that $A,B,C,D,S,T,\{a,b,c,d\}$ is a partition of $V_1\cup V_2$. The subsets are pairwise disjoint. By contradiction, we assume that there exists $v\in V_1$ that is not in one of the previous subsets. By Fact \ref{factdiam3bip} $v$ and $a$ have a common neighbor $w$.  Since $N(a)\subseteq A\cup \{b,c\}$ we have a contradiction. The case $v\in V_2$ is the same.\\

By Fact \ref{neighmat} we have $A\cup C\subseteq X$ and $B\cup D\subseteq Y$. Let $v\in A$. By Fact \ref{twoneigh}, if $v$ has two neighbors in $B$ then $v\in Y$, so $M$ cannot exist. The situation is the same when a vertex of $B$ has two neighbors in $A$, a vertex of $C$ has two neighbors in $D$, a vertex of $D$ has two neighbors in $C$. If $v$ has exactly one neighbor $w\in B$ then $vw\in M$ and by Fact \ref{neighmat} all its neighbors are put in $X$, and all the neighbors of $w$ are put in $Y$. We do  the same for the vertices of $B,C,D$. By Fact \ref{twoneigh} when $v\in S$ has two neighbors in $X$, resp. $Y$, then $v\in X$, resp. $v\in Y$. We do in a same way for $v\in T$. If a vertex is in both $X$ and $Y$ then $M$ cannot exist and we stop. By Fact \ref{twoneigh} if a vertex in $X$, resp. $Y$, has two neighbors in $Y$, resp. $X$, then $M$ cannot exist. \\

Let $S'=\{v\in  S \mid v\not\in X\cup Y\}$ and $T'=\{v\in  T \mid v\not\in X\cup Y\}$. By above and since $\delta(G)\ge 2$, each vertex $v\in S'$ has exactly one neighbor $v_a\in A$ and one neighbor $v_d\in D$, and each vertex $v\in T'$ has exactly one neighbor  $v_b\in B$ and one neighbor $v_c\in C$. Let $A'=\{v_a\in A \mid vv_a\in E,\, v\in S'\},\, D'=\{v_d\in D \mid vv_d\in E,\, v\in S'\},\, B'=\{v_b\in B\mid vv_b\in E,\, v\in T'\},\, C'=\{v_c\in C\mid vv_c\in E,\, v\in T'\}$. Note that from Fact \ref{twoneigh}, for every pair $v_a\in A',\, v_d\in D'$ we have $N(v_a)\cap N(v_d)\subseteq S'$. By symmetry, for every pair $v_b\in B',\, v_d\in D'$ we have $N(v_b)\cap N(v_c)\subseteq T'$. \\

For every $v\in S'$, resp. $v\in T'$, for $M$ to exist we have either $vv_a\in M$ or $vv_d\in M$, resp. $vv_b\in M$ or $vv_c\in M$. Hence every edge $st,s\in S',\, t\in T'$ is such that $st\not\in M$ and the two vertices $s,t$ will be assigned to a same subset $X$ or $Y$. \\

Let $\vert S'\vert=\mu,\, \vert A'\vert=\alpha,\, \vert D'\vert=\delta$. W.l.o.g we assume that $\alpha\ge \delta$ (the case $\alpha\le \delta$ being symmetric). By Fact \ref{factdiam3bip} and since each vertex of $S'$ has exactly one neighbor in $A'$ and one neighbor in $D'$, we have $\mu\ge\alpha\delta$. For $M$ to exist we need $\mu\leq \alpha+\delta$. Thus $ \alpha+\delta\ge \alpha\delta$, which is possible only for $\alpha=\delta=2$ or $\delta =1,\, \alpha\ge 1$. \\

Let $\alpha=\delta=2$. We denote $A'=\{v_a^1,v_a^2\},\, D'=\{v_d^1,v_d^2\},\, S'=\{v_1,v_2,v_3,v_4\}$. Then $G'=G[A'\cup D'\cup S']$ consists of the four paths $v_a^1-v_1-v_d^1$, $v_a^1-v_2-v_d^2$, $v_a^2-v_3-v_d^1$, $v_a^2-v_4-v_d^2$. There exist two perfect matchings of $G'$, that are, $M_a=\{v_a^1v_1,v_d^2v_2,v_d^1v_3,v_a^2v_4\},\, M_d=\{v_a^1v_2,v_d^2v_4,v_d^1v_1,v_a^2v_3\}$. \\

Let $\delta=1$. We have $\alpha\leq \mu\leq \alpha+1$. We denote $A'=\{v_a^1,\ldots,v_a^\alpha\},\, D'=\{v_d\},\, S'=\{v_1,\ldots,v_\mu\}$ and we assume that $v_a^iv_i\in E,1\le i\le\mu$. We denote  $G'=G[A'\cup D'\cup S']$.

First $\mu=\alpha$. Then $G'$ consists of $\alpha$ paths $v^1_a-v_1-v_d,\ldots,v^\alpha_a-v_\alpha-v_d$. Note that $G'$ has no perfect matching but recall that for each $v_i\in S'$, either $v_iv_a^i\in M$ or $v_iv_d\in M$. So in $G'$ there exists $\alpha+1$ matchings that disconnect $A'$ from $D'$.
These matchings are $M_0=\{v_a^1v_1,\ldots, v_a^{\alpha}v_\alpha\}$ and $M_i=\{v_iv_d\}\cup\{v_a^jv_j,\,1\le j\le \alpha,\, j\ne i\}$.

Second $\mu=\alpha +1$. Then $G'$ consists of the two paths $v_a^1-v_1-v_d,\, v_a^1-v_2-v_d$ and the $\alpha-1$ paths $v^2_a-v_3-v_d,\ldots,v^\alpha_a-v_{\alpha+1}-v_d$. There exists exactly two (perfect) matchings of $G'$ that disconnect $A'$ from $D'$, that are ${\bar M_1}=\{v_a^1v_1,v_dv_2,v_a^2v_3,\ldots,v^\alpha_av_{\alpha+1}\}$ and ${\bar M_2}=\{v_a^1v_2,v_dv_1,v_a^2v_3,\ldots,v^\alpha_av_{\alpha+1}\}$. \\

By symmetry, to $G''=G[B'\cup C'\cup T']$ correspond  the following matchings: either $M'_0=\{{v'}_a^1{v'}_1,\ldots, {v'}_a^{\alpha'}{v'}_{\alpha'}\}$ and $M'_i=\{v'_iv'_d\}\cup\{{v'}_a^jv'_j,1\le j\le \alpha',j\ne i\}$ or ${\bar M'_1}=\{{v'}_a^1v'_1,v'_dv'_2,{v'}_a^2v'_3,\ldots,{v'}^{\alpha'}_a-v'_{\alpha'+1}\}$ and ${\bar M'_2}=\{{v'}_a^1v'_2,v'_dv'_1,{v'}_a^2v'_3,\ldots,{v'}^{\alpha'}_a-v'_{\alpha'+1}\}$.
Hence there are $\mathcal{O}(\vert E\vert^2)$ combinations between the matchings of $G'$ and $G''$.

For each combination we test if $E(X,Y)$ is a matching-cut. If not then $M $ with $E(X,Y)\subseteq M$ cannot exist. Otherwise, let $X'\subseteq X$ such that $N(X')\cap Y=\emptyset$ and $Y'\subseteq Y$ such that $N(Y')\cap X=\emptyset$. We check if $G[X'\cup Y']$ has a perfect matching. If not, $M$ with $E(X,Y)\subseteq M$ cannot exist, else we have $M$ a perfect matching-cut of $G$.\\

We estimate the running time of our algorithm as follows. From \cite{Hopcroft}, we know that computing a perfect matching in a bipartite graph takes $\mathcal{O}(\vert V\vert^{\frac 5 2})$. To check if there exists a perfect matching that contains a bridge can be done in time $\mathcal{O}(\vert V\vert^{\frac 5 2})$. Now, there are $\mathcal{O}(\vert V\vert^2)$ pairs of edges $ab,cd$. Given a pair $ab,cd$, one can verify that the running time until the next pair is $\mathcal{O}(\vert V\vert^{\frac 5 2})$. Hence the complexity of the algorithm is $\mathcal{O}(\vert V\vert^{\frac 9 2})$.
\end{proof}

Now we give our result for the graphs with a diameter two.
\begin{theorem}\label{diam2}
   Let $G=(V,E)$ be a graph with $diam(G)=2$. Deciding whether  $G$ has a perfect matching-cut can be done in polynomial time.
\end{theorem}
\begin{proof}
   We can assume that $G$ has a perfect matching. Let $xy$ be an edge. We check if there exists a matching-cut $E(X,Y)$ such that $x\in X$ and $y\in Y$. From Fact \ref{neighmat}, we set $X\leftarrow X\cup N(X)\setminus \{y\}$ and $Y\leftarrow Y\cup N(Y)\setminus \{x\}$. If $X\cap Y\neq \emptyset$ or if $E(X,Y)$ is not a matching then there is no matching-cut $E(A,B)$ with $x\in A$, $y\in B$. Else, since $diam(G)=2$, we have $X,Y$ a partition of $V$. Let $W$ be the set of vertices with no endpoint in $M=E(X,Y)$. If there is a perfect matching $M'$ in $G[W]$ then $M\cup M'$ is a perfect matching-cut of $G$. Else there is no perfect matching-cut that contains the cut $M$.
\end{proof}

\begin{rmk}
The cliques are the graphs with diameter one. Hence $K_2$ is the sole graph of diameter one that has a perfect matching-cut.
\end{rmk}

\section{Planar graphs}\label{planar}
 We show the NP-completeness of the Perfect Matching-Cut problem for two subclasses of planar graphs. The first is when the degrees are restricted to be three or four, the second when the girth of the graph is five.\\

For our first proof we use the following decision problem.

\begin{center}
   \begin{boxedminipage}{.99\textwidth}
   \textsc{\sc Segment 3-Colorability} \\[2pt]
   \begin{tabular}{ r p{0.8\textwidth}}
   \textit{~~~~Instance:} & A set of vertices $V$ and three disjoint edge sets $A, B, C$ such that
   $G = (V, A\cup B\cup C)$ is a $3$-regular planar multigraph with Hamilton cycle $A\cup B$.\\
   \textit{Question:} &  Is there a color function $f : V \rightarrow \{1, 2, 3\}$ such that if $uv\in A$ then $f(u) = f(v)$ and if $uv\in C$ then $f(u)\neq f(v)$ ?
   \end{tabular}
   \end{boxedminipage}
\end{center}

\begin{theorem}[\cite{Bonsma}]
   Segment $3$-Colorability is NP-complete.
\end{theorem}

The proof of the following theorem is inspired by the one of P. Bonsma for the Matching-Cut problem for planar graphs in \cite{Bonsma}.
\begin{theorem}\label{planarmax4}
   Perfect Matching-Cut is NP-complete when restricted to planar graphs with vertex degrees $\delta(v)\in \{3,4\}$.
\end{theorem}
\begin{proof}
Let $(V, A, B, C)$ be an instance of Segment $3$-Colorability with $G = (V, A \cup B\cup C)$. We will transform, in polynomial time, this instance into an instance $G'$ of Perfect Matching-Cut where $G'$ is planar with maximum degree four. Then we will transform $G'$ into a planar graph $G''$ with vertex degrees three or four.

Every vertex $v$ of $G$ will be associated with a connected component $g_v$ in $G'$ as shown in Figure \ref{gadgetvertex}. We often call such connected component a vertex $g_v$ of $G'$. If two vertices $u$ and $v$ of $G$ are joined by an $A$-edge $e$ then the vertices $g_u$ and $g_v$ will be linked using the connected component $a_e$ as shown in Figure \ref{gadgetedgeA}. We often call such connected component an edge $a_e$ of $G'$. If two vertices $u$ and $v$ of $G$ are joined by a $B$-edge $e$ then the vertices $g_u$ and $g_v$ will be linked using the connected component $b_e$ as shown in Figure \ref{gadgetedgeB}. We often call such connected component an edge $b_e$ of $G'$. If two vertices $u$ and $v$ of $G$ are joined by a $C$-edge $e$ then the vertices $g_u$ and $g_v$ will be linked by six disjoint paths of size three ($c_{u,i}-d_i-c_{v,i}$ and $c_{u,i}'-d_i'-c_{v,i}'$, $i=1,2,3$) as shown in Figure \ref{gadgetedgeC}. We say that $g_u$ and $g_v$ are connected by a $C$-edge.

\begin{figure}[H]
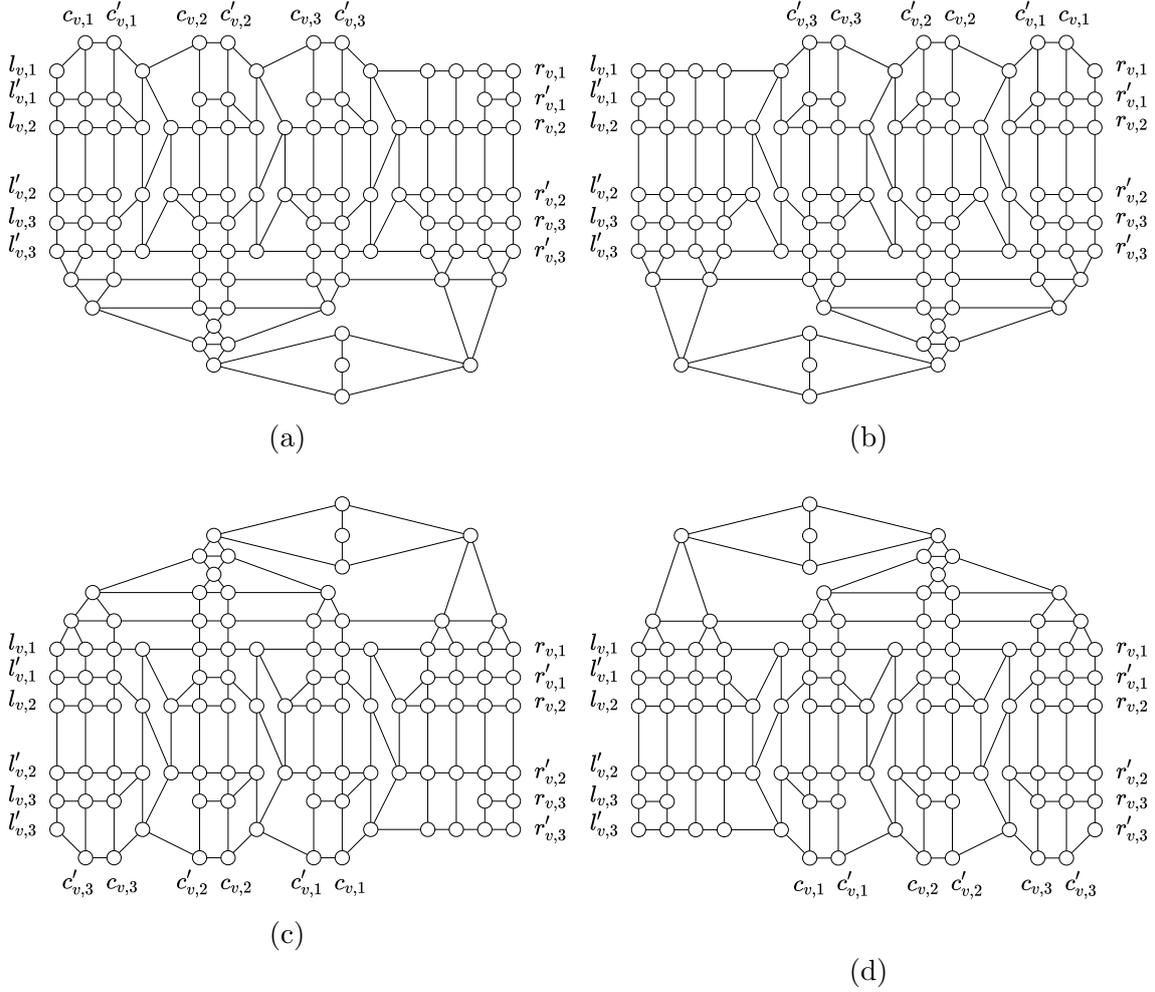

   \centering
   \begin{subfigure}{.49\textwidth}
      \centering
      \includegraphics[width=1\linewidth]{gadget_vertex1}
      \caption{}
      \label{gadgetvertex1}
   \end{subfigure}\hfill%
   \begin{subfigure}{.49\textwidth}
      \centering
      \includegraphics[width=1\linewidth]{gadget_vertex2}
      \caption{}
      \label{gadgetvertex2}
   \end{subfigure}
   \bigskip
   \begin{subfigure}{.49\textwidth}
      \centering
      \includegraphics[width=1\linewidth]{gadget_vertex3}
      \caption{}
      \label{gadgetvertex3}
   \end{subfigure}\hfill%
   \begin{subfigure}{.49\textwidth}
      \begin{center}
         \includegraphics[width=1\linewidth]{gadget_vertex4}
      \end{center}
      \caption{}
      \label{gadgetvertex4}
   \end{subfigure}
   \caption{Four orientations of a vertex $g_v$ (a) tail of inside edge, (b) head of inside edge, (c) tail of outside edge and (d) head of outside edge.}
   \label{gadgetvertex}
\end{figure}

\begin{figure}
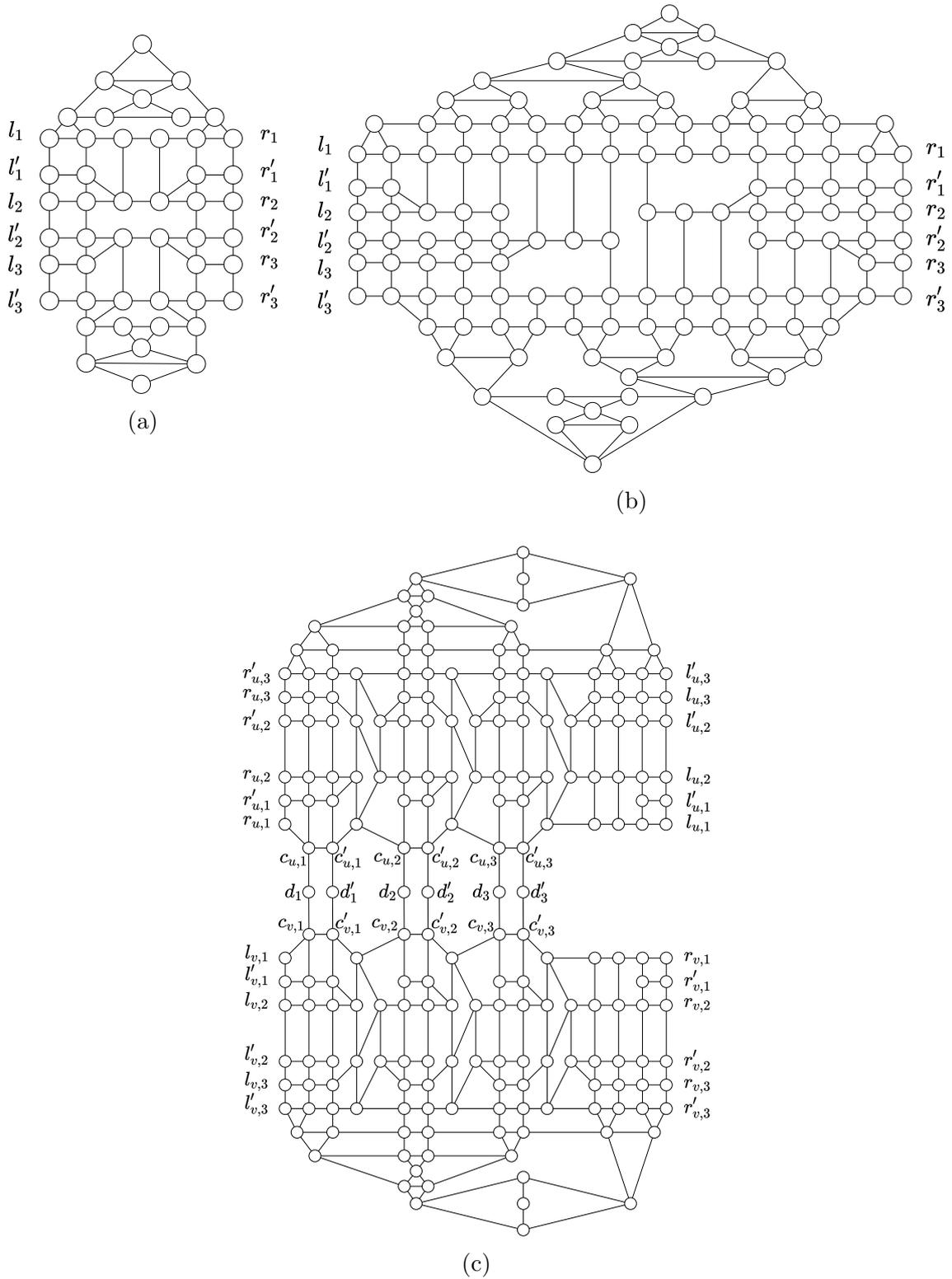

   \centering
   \begin{subfigure}{.29\textwidth}
      \centering
      \includegraphics[width=1\linewidth]{gadget_A}
      \caption{}
      \label{gadgetedgeA}
   \end{subfigure}\hfill%
   \begin{subfigure}{.67\textwidth}
      \centering
      \includegraphics[width=1\linewidth]{gadget_B}
      \caption{}
      \label{gadgetedgeB}
   \end{subfigure}

   \bigskip
   \begin{subfigure}{.5\textwidth}
      \centering
      \includegraphics[width=1\linewidth]{gadget_C}
      \caption{}
      \label{gadgetedgeC}
   \end{subfigure}
   \caption{The three associated connection (a) the  component $a_e$ associated with an $A$-edge $e$ of $G$, (b) the component $b_e$ associated with a $B$-edge $e$ of $G$ and (c) two component $g_u,g_v$ where $uv\in C$}
   \label{gadgetedges}
\end{figure}

Consider an embedding of $G$. Orient all edges of the Hamilton cycle $A\cup B$ counterclockwise with regard to the embedding. We say that a vertex $u$ precedes $v$ if there is an edge from $u$ to $v$ in the choosen embedding. Orient all edges of $C$ arbitrarily. Since the edge set $A\cup B$ gives a Hamilton cycle, in the embedding of $G$, this Hamilton cycle divides the plane into two regions. Hence we can divide the edges of $C$ into two categories using this Hamilton cycle: inside edges and outside edges. Because every vertex $v$ of $G$ is incident with exactly one edge $e$ of $C$, we can use this to define four variants of a vertex $g_v$:

\begin{enumerate}
   \item If $e$ is an inside edge and $v$ is incident with the tail of $e$, introduce $g_v$ as shown in Figure \ref{gadgetvertex1};
   \item If $e$ is an inside edge and $v$ is incident with the head of $e$, introduce $g_v$ as shown in Figure \ref{gadgetvertex2};
   \item If $e$ is an outside edge and $v$ is incident with the tail of $e$, introduce $g_v$ as shown in Figure \ref{gadgetvertex3};
   \item If $e$ is an outside edge and $v$ is incident with the head of $e$, introduce $g_v$ as shown in Figure \ref{gadgetvertex4}.
\end{enumerate}

For each component used to construct $G'$, we will show that the only possible matching-cuts are those depicted in Figures \ref{gadgetvertexci}, \ref{gadgetedgeAci} and \ref{gadgetedgeBcij}. Note that there exists two special cases not drawn in the figures, concerning the component $g_v$, but these cases are impossible when the components will be linked together in $G'$. We will highlight these two cases later in the proof.
First, for each component, we will focus on the matching-cuts such that the number of edges is minimal when one edge is forced to be in the cut. The minimal matching-cuts considered are those that force exactly one edge at the left or at the right extremity of the component representation (the edges between the $l$'s vertices or between the $r$'s vertices). We will show that the bold edges crossed by the dotted lines in Figure \ref{gadgetvertexci} are mandatory to have a matching-cut. If not specified, when we refer to the matching-cut of a component we refer only to these mandatory edges. Then for each configuration, we will derive a perfect matching-cut that is represented by all the other bold edges of the figures. Note that the remaining edges that form the perfect matching-cut will not add any cuts other than the initial one.

Before showing that the only matching-cuts through the components are those outlined in Figure \ref{gadgetvertexci}, we point out the matching-cut properties of some of the induced subgraphs. Recall that a triangle is immune, and therefore no edge of a triangle belongs to a minimal matching-cut. Consider now a $C_4$. Taking an edge of the $C_4$ into a matching-cut, there is exactly one other edge of $C_4$ into the matching-cut, that is, the one that is not incident with the former. These properties able us to control the matching-cut from the left to the right for each component. Moreover, they avoid cuts to start from the bottom or the top. Hence we only consider the edges at the left or at the right of the component when we initialize a minimal matching-cut.

We focus on the vertex $g_v$ and the three matching-cuts shown in Figure \ref{gadgetvertexci}. We show that they are the only possible matching-cuts. Note that the matching-cuts outlined by \ref{gadgetvertexc1} and \ref{gadgetvertexc2} could be modify so they go through $r_{v,1}'r_{v,2}$ or $l_{v,2}'l_{v,3}$, respectively. Yet, note that every matching-cut of $g_v$ containing $r_{v,1}'r_{v,2}$ or $l_{v,2}'l_{v,3}$, is not valid when $g_v$ is connected to component $a_e$ or $b_e$, since there are blocked by a triangle. There exist some cycles of length more than four in $g_v$. Starting from an edge of one of them, it is easy to check that there is no matching-cut except these depicted in Figure \ref{gadgetvertexci}  (in any case we end up in a triangle).
Hence we focus only on  $C_3$ and $C_4$. Then taking an edge to initialize a cut, it is easy to construct a matching-cut since most of the configurations will end up in a triangle, and the path we build is constrained by a sequence of $C_4$'s. The three ways to build a matching-cut are outlined by Figure \ref{gadgetvertexci}. Now it remains to construct a perfect matching-cut. Consider a matching-cut outlined in Figures \ref{gadgetvertexci}. By taking the remaining bold edges,  we can complete the matching-cut so that it becomes a perfect matching-cut.

When $l_{v,i}l_{v,i}'$ is in a matching-cut of $g_v$, we say that $g_v$ is {\it cut} by $i$ $(i=1,2,3)$. Every choice of $i,\, i\in\{i=1,2,3\}$, corresponds to one of the three colors of the vertex $v$ in the original graph $G$. Observe that in this case $r_{v,i}r_{v,i}'$ is also in the matching-cut of $g_v$. Moreover it is relevant to note that $c_{v,i}c_{v,i}'$ is also in the matching-cut whereas $c_{v,j}c_{v,j}'$ is not $(i\neq j)$. \\

\begin{figure}[H]
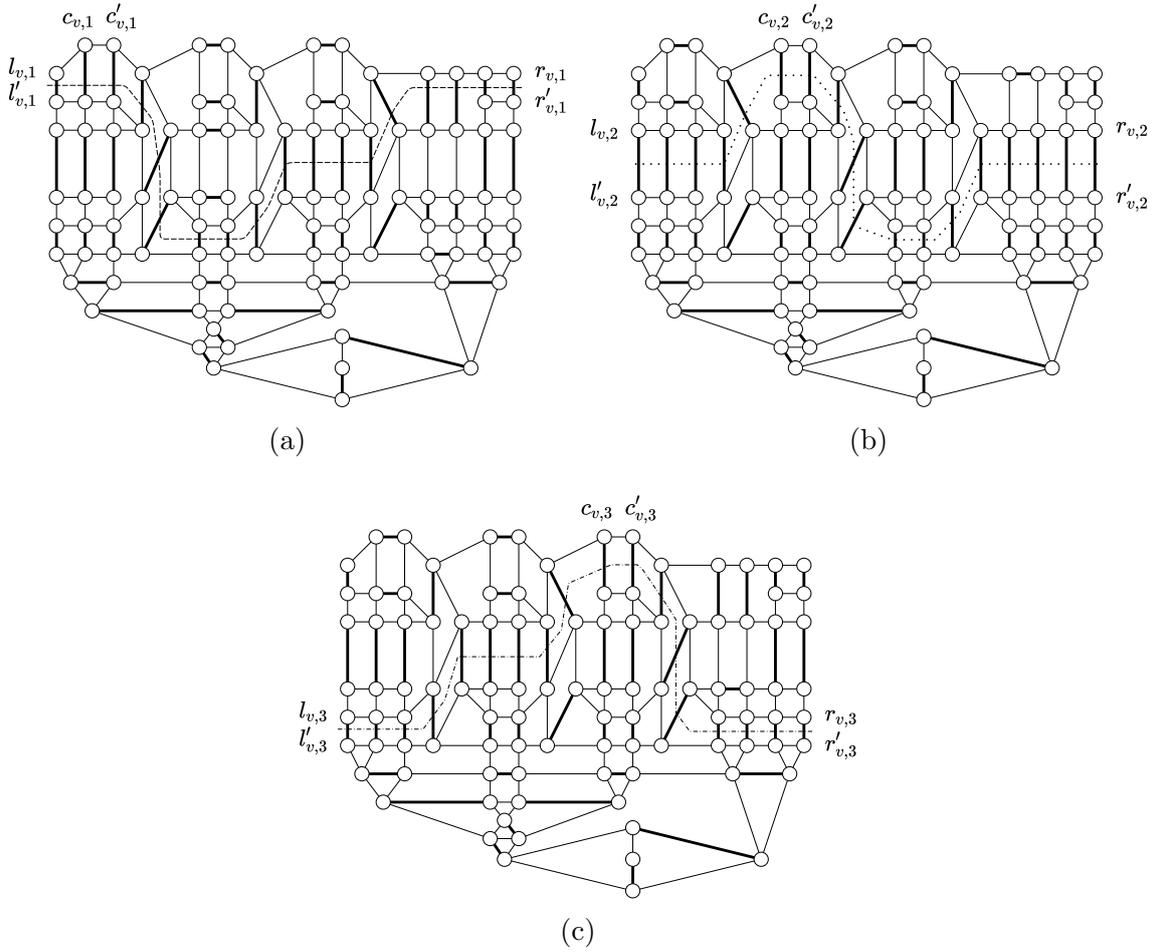

   \centering
   \begin{subfigure}{.49\textwidth}
     \centering
     \includegraphics[width=1\linewidth]{gadget_vertex_c1}
      \caption{}
      \label{gadgetvertexc1}
   \end{subfigure}\hfill%
   \begin{subfigure}{.49\textwidth}
     \centering
     \includegraphics[width=1\linewidth]{gadget_vertex_c2}
      \caption{}
      \label{gadgetvertexc2}
   \end{subfigure}

   \bigskip
   \begin{subfigure}{.49\textwidth}
      \centering
      \includegraphics[width=1\linewidth]{gadget_vertex_c3}
      \caption{}
      \label{gadgetvertexc3}
   \end{subfigure}
   \caption{The three  matching-cuts through a vertex $g_v$ outlined by the dotted lines. The bold edges represent the corresponding perfect matching-cut.}
   \label{gadgetvertexci}
\end{figure}

Now we focus on the edges $A,B,C$. We study how they are connected to the vertices of $G'$.
Each edge $e$ of $A$ is replaced by an edge $a_e$ as shown in Figure \ref{gadgetedgeA}. Three matching-cuts are outlined by the Figure \ref{gadgetedgeAci}. We show that they are the only three matching-cuts. Note that there is no matching-cut passing through $l_i'l_{i+1}$ or $r_i'r_{i+1}$, $(i=1,2)$ since we would end up with a triangle. Moreover, one can check that there is no matching-cut passing through the upper or the lower side of a component $a_e$ because of the triangles. Last, consider the $C_8$ at the center of the component. Because of the positions of the four triangles and the four squares that surround $C_8$, a matching-cut goes entirely either horizontally or vertically through it.

As for a vertex $g_v$, if an edge $l_1l_1'$ is in a matching-cut then $r_1r_1'$ is in the  matching-cut. Hence we obtain the following property:

\begin{prop}\label{propedgeA}
   If two vertices $g_u$ and $g_v$ are connected by an edge $a_e$, then for any matching-cut we have that $g_u$ is cut by $i$ if and only if $g_v$ is cut by $i$ $(i=1,2,3)$.
\end{prop}

Note that to achieve this property, the vertex $g_v$ would be sufficient. Yet, by the way we link the components in $G'$, it would have vertices of degree greater than four. Also, we could merge $g_v$ with $a_e$ to form a unique component, but for the readability of the proof, we prefer to minimize the size of our components.

Now it remains to construct a perfect matching-cut of an edge $a_e$. The dotted lines and the bold edges depicted in Figure \ref{gadgetedgeAci} show how to get three perfect matching-cuts of $a_e$.

\begin{figure}[H]
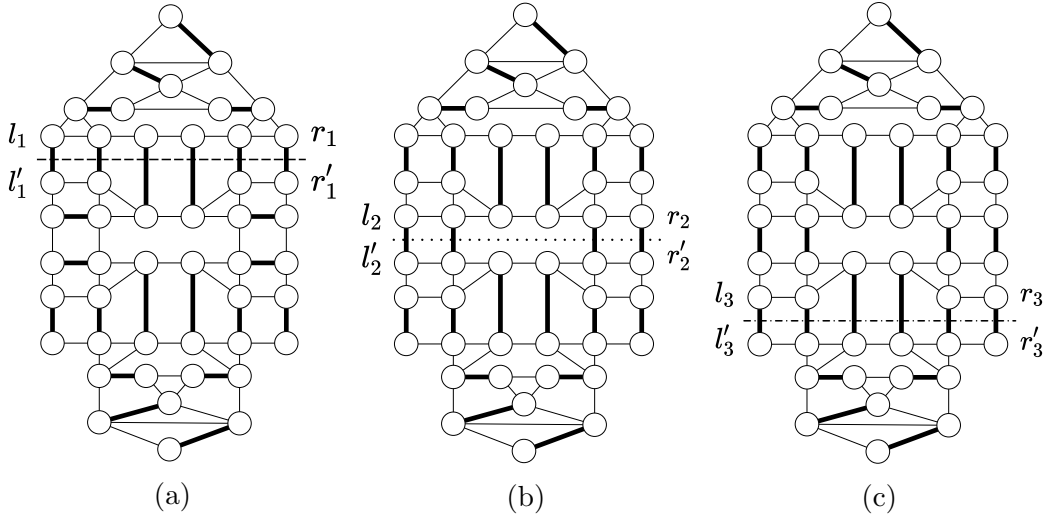

   \centering
   \begin{subfigure}{.29\textwidth}
     \centering
     \includegraphics[width=1\linewidth]{gadget_A_c1}
      \caption{}
      \label{gadgetedgeAc1}
   \end{subfigure}\hspace{0.02\textwidth}%
   \begin{subfigure}{.29\textwidth}
     \centering
     \includegraphics[width=1\linewidth]{gadget_A_c2}
      \caption{}
      \label{gadgetedgeAc2}
   \end{subfigure}\hspace{0.02\textwidth}%
   \begin{subfigure}{.29\textwidth}
      \centering
      \includegraphics[width=1\linewidth]{gadget_A_c3}
      \caption{}
      \label{gadgetedgeAc3}
   \end{subfigure}
   \caption{The three matching-cuts through an edge $a_e$ outlined by the dotted lines. The bold edges represent the corresponding perfect matching-cut.}
   \label{gadgetedgeAci}
\end{figure}

Each edge $e$ of $B$ is replaced by an edge $b_e$ as shown in Figure \ref{gadgetedgeB} and nine matching-cuts are outlined in the Figure \ref{gadgetedgeBcij}. We show that another matching-cut is not possible.
First, for each of the three figures exposed in \ref{gadgetedgeBcij}, consider $H$ the subgraph induced by the vertices not covered by a bold edge. Note that $H$ can be partitioned into two isomorphic subgraphs $C$ and $C'$.
We focus on the subgraph $C$ which is displayed in Figure \ref{gadgetedgeBci}. Suppose that there is no cut passing through the edges of the top border or of the lower border of $C$. With similar arguments used for the cuts of $g_v$ and $a_e$, one can check that the only three  cuts are outlined in Figures \ref{gadgetvertexc1}, \ref{gadgetvertexc2} and \ref{gadgetvertexc3}. Also, as for $a_e$, there is no matching-cut containing $l_i'l_{i+1}$ or $r_i'r_{i+1}$, $(i=1,2)$ because of the triangles.
Now consider the entire graph $b_e$. One can check that there is no matching-cut passing through the upper side or the lower side of the graph, because of the triangles. Hence, no matching-cut can goes through $C$ or $C'$ by passing through their bottom or their top. Note that each of the three cuts of $C$  is compatible with the three cuts of $C'$. So there are nine matching-cuts which are outlined in Figure \ref{gadgetedgeBcij}.

If an edge $l_il_i'$, $(i=1,2,3)$ is in a matching-cut then $r_jr_j'$, $(j=1,2,3)$ is in the matching-cut. Hence we obtain the following property:

\begin{prop}\label{propedgeB}
   If two vertices $g_u$ and $g_v$ are connected by an edge $b_e$ then for every matching-cut we have that $g_u$ is cut if and only if $g_v$ is cut. Every combination of cuts through $g_u$ and $g_v$ is possible.
\end{prop}

\begin{figure}[H]
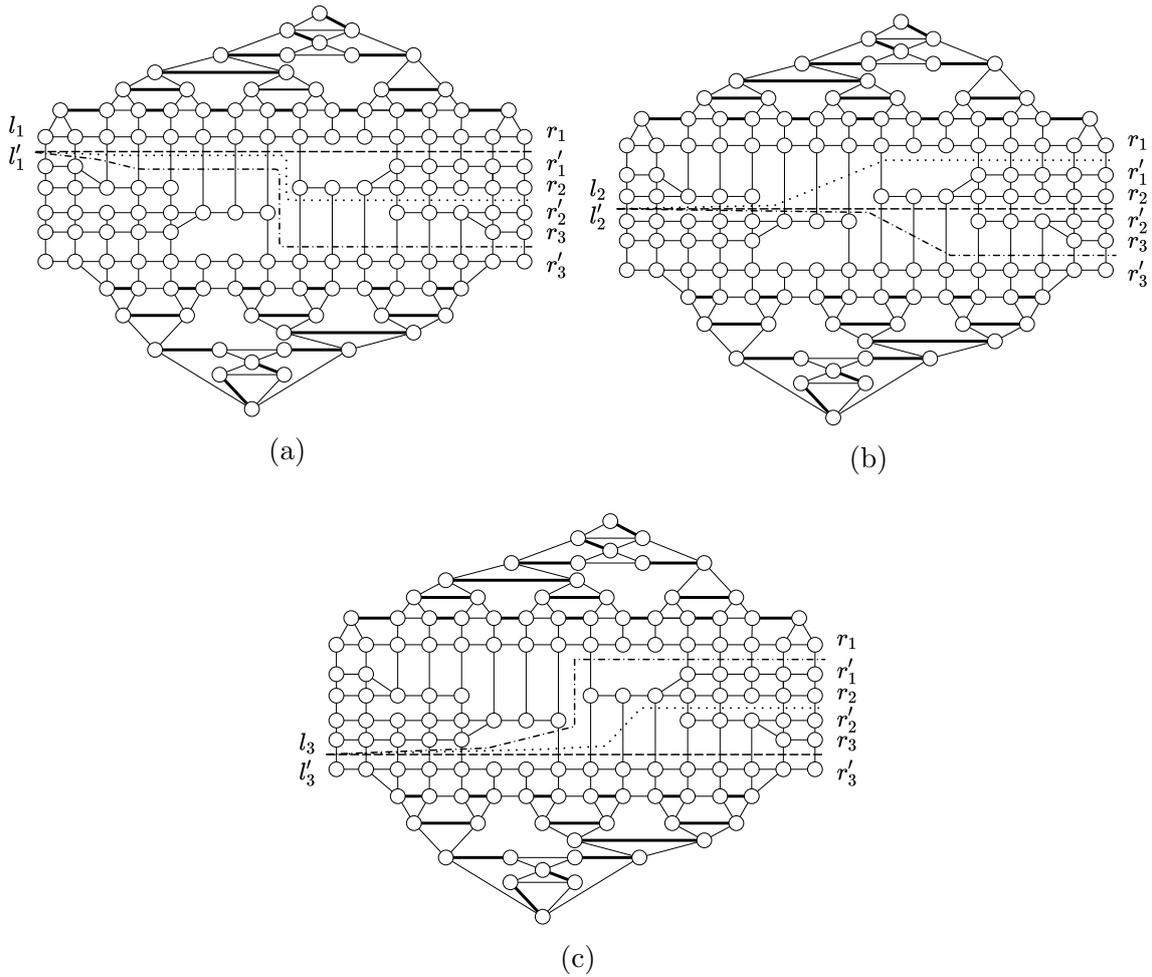

   \centering
   \begin{subfigure}{.49\textwidth}
     \centering
     \includegraphics[width=1\linewidth]{gadget_B_c1j}
      \caption{}
      \label{gadgetedgeBc1j}
   \end{subfigure}\hfill%
   \begin{subfigure}{.49\textwidth}
     \centering
     \includegraphics[width=1\linewidth]{gadget_B_c2j}
      \caption{}
      \label{gadgetedgeBc2j}
   \end{subfigure}

   \bigskip
   \begin{subfigure}{.49\textwidth}
      \centering
      \includegraphics[width=1\linewidth]{gadget_B_c3j}
      \caption{}
      \label{gadgetedgeBc3j}
   \end{subfigure}
   \caption{The nine  matching-cuts through an edge $b_e$ outlined by the dotted lines. The bold edges represent a part of the corresponding perfect matching-cut.}
   \label{gadgetedgeBcij}
\end{figure}

\begin{figure}[H]
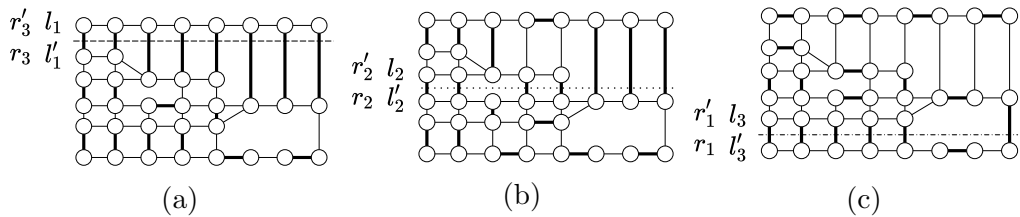

   \centering
   \begin{subfigure}{.3\textwidth}
     \centering
     \includegraphics[width=1\linewidth]{gadget_B_c1}
      \caption{}
      \label{gadgetedgeBc1}
   \end{subfigure}%
   \begin{subfigure}{.3\textwidth}
     \centering
     \includegraphics[width=1\linewidth]{gadget_B_c2}
      \caption{}
      \label{gadgetedgeBc2}
   \end{subfigure}%
   \begin{subfigure}{.3\textwidth}
      \centering
      \includegraphics[width=1\linewidth]{gadget_B_c3}
      \caption{}
      \label{gadgetedgeBc3}
   \end{subfigure}
   \caption{The three matching-cuts of the component $C$ of an edge $b_e$ outlined by the dotted lines. The bold edges represent the corresponding perfect matching-cut.}
   \label{gadgetedgeBci}
\end{figure}

Now it remains to construct a perfect matching-cut of an edge $b_e$. For every graph of Figure \ref{gadgetedgeBcij}, first we take the bold edges. Then we consider one of the matching-cut. From this cut, we take the edges that are cut in the two subgraphs $C$ and $C'$ in Figure \ref{gadgetedgeBci}. Eventually, we take the rest of the bold edges that are not cut. So we have a perfect matching-cut of $b_e$. \\

We show how two vertices $g_u$ and $g_v$ are connected in $G'$ when there exists an edge $uv\in A\cup B$ where $u$ precedes $v$. We merge each vertex $r_{u,i}$, $r_{u,i}'$ with the corresponding vertex $l_i$, $l_i'$. Then we merge each vertex $l_{v,i},l_{v,i}'$ with the corresponding vertex $r_i,r_i'$. Note that both operations do not create a vertex of degree more than four. Moreover, recall that there is no matching-cut passing through $l_i'l_{i+1}$ or $r_i'r_{i+1}$, $(i=1,2)$ for the edges $a_e$ and $b_e$. Hence we do not have to consider the matching-cut ofs $g_u$ that go through $l_{u,i}'l_{u,i+1}$ or $r_{u,i}'r_{u,i+1}$, $i=1,2$. Thus we consider the three  matching-cuts of $g_u$ depicted in Figure \ref{gadgetvertexci}. \\

Last we show how two vertices $g_u$ and $g_v$ are connected in $G'$ when there exists an edge $uv\in C$ where $u$ precedes $v$. We link each $c_{u,i}$ to $c_{v,i}$ and each $c_{u,i}'$ to $c_{v,i}'$ with the respective two-path $c_{u,i}-d_i-c_{v,i}$ and $c_{u,i}'-d_i'-c_{v,i}'$, $(i=1,2,3)$ has shown in Figure \ref{gadgetedgeC}. Note that this construction does not create a vertex with a degree more than four. We prove the following:

\begin{prop}\label{propedgeC}
   If two vertices $g_u$ and $g_v$ are connected by a $C$-edge then for every matching-cut we have that  if $g_u$ is cut by $i$ then $g_v$ is cut by $j$, $i\neq j$.
\end{prop}
\begin{proof}
   We know that if a vertex $g_u$ is cut by $i$ then $c_{u,i}$ is in the matching-cut of $g_u$. Suppose now that $g_v$ is cut by $i$. Thus $c_{v,i}$ is in the matching-cut of $g_v$. Since $c_{u,i}$ is connected by a two-path, there exists $d_i$ with degree two that is not in a matching, but it has both neighbors in a matching. Hence there is no perfect matching-cut such that $g_u$ and $g_v$ are cut by $i$.
\end{proof}

Let $G'$ be the graph that is constructed by introducing a vertex $g_u$ for each vertex $u$ of $G$ and connecting them as described above. Using the observed properties of these connections, we prove the following:

\begin{prop}
   Every matching-cut in $G'$ corresponds to a proper segment 3-coloring of $G$ and vice versa.
\end{prop}
\begin{proof}
   Because $A\cup B$ is a Hamilton cycle in $G$, and the components have no matching-cuts other than the indicated, by Property \ref{propedgeA} and Property \ref{propedgeB}, we conclude that if $G'$ has a matching-cut, this matching-cut disconnects every component associated with the vertices of $G$, and every component associated with the edges of $A\cup B$. Note that since the $C$-edges cannot disconnect the graph $G'$ then for every $C$-edge linking two vertices $g_u$ and $g_v$, a minimal matching-cut cannot contain some of the edges  $c_{u,i}-d_i-c_{v,i}$, $c_{u,i}'-d_i'-c_{v,i}',i=1,2,3$.
\end{proof}

We focus now on the perfect matching-cuts of $G'$ and demonstrate the following:

\begin{prop}
   Given a matching-cut of $G'$ we can construct a perfect matching-cut of $G'$.
\end{prop}
\begin{proof}
   Let $M$ be a matching-cut of a vertex $g_u$, and let an edge $a_e$ or $b_e$. We  show how we  construct a perfect matching-cut $M'$ such that $M\subset M'$. We start by taking the only necessary vertices for the cut $i$ to be minimal. Then for the vertices $l_{u,j}$, $l_{u,j}'$ $(i\neq j)$ that are merged with an edge $a_e$ or $b_e$, we consider the perfect matching-cut of $a_e$ or $b_e$ outlined in the Figures \ref{gadgetedgeAci}, \ref{gadgetedgeBcij} and \ref{gadgetedgeBci}. Next, we consider $g_v$ such that $g_u$ and $g_v$ are connected by a $C$-edge. We can assume that $u$ precedes $v$. Since $g_v$ is cut by $j$ (recall that $i\neq j$), we take $c_{u,j}d_j$, $c_{u,j}'d_j'$, $c_{v,i}d_i$, $c_{v,i}'d_i',c_{u,k}d_k$, $c_{u,k}'d_k'$ $(i\neq j\neq k)$ in the matching-cut. Actually, for every vertex $g_u$, it is clear from Figure \ref{gadgetvertexci} that we take every vertex not covered by the matching-cut. So each $g_u$, $a_e$ and $b_e$ is cut by a perfect matching-cut. Moreover all vertices associated with a $C$-edge are also in a matching. Hence we have a perfect matching-cut of $G'$.
\end{proof}

At last, we prove is the following.
\begin{prop}
   $G'$ is planar and $\Delta(G')=4$.
\end{prop}
\begin{proof}
   Clearly all the components replacing the vertices and the edges are planar with maximum degree four, see Figure \ref{gadgetvertex} and \ref{gadgetedges}. Moreover, we have outlined that the way we link the components does not produce a vertex with a degree more than four. Hence $\Delta(G')=4$. Because $A\cup B$ is a Hamilton cycle in $G$, we can one by one make the links between each vertex $g_u$ without destroying the planarity. It remains the $C$-edge links. Given a counterclockwise orientation of the Hamilton cycle $A\cup B$ it is clear that no $C$-edges overlap in the graph $G$, recall that they are either inside or outside the hamilton cycle. Let  $g_u$ and $g_v$ be two vertices connected by a $C$-edge. If they are connected by an inside edge, we can see from Figure \ref{gadgetvertex1} and \ref{gadgetvertex2} that the three paths connections between them does not overlap. If $g_u$ and $g_v$ are connected by an outside edge, the same property can be checked, see Figure \ref{gadgetvertex3} and \ref{gadgetvertex4}. Hence the graph $G'$ is planar and has degree maximum four.
\end{proof}

The last property completes the proof, that is, $G'$ has a perfect matching-cut if and only if $G$ is segment $3$-colorable.
\begin{figure}[htbp]
   \begin{center}
   \includegraphics[width=15cm, height=3.5cm, keepaspectratio=true]{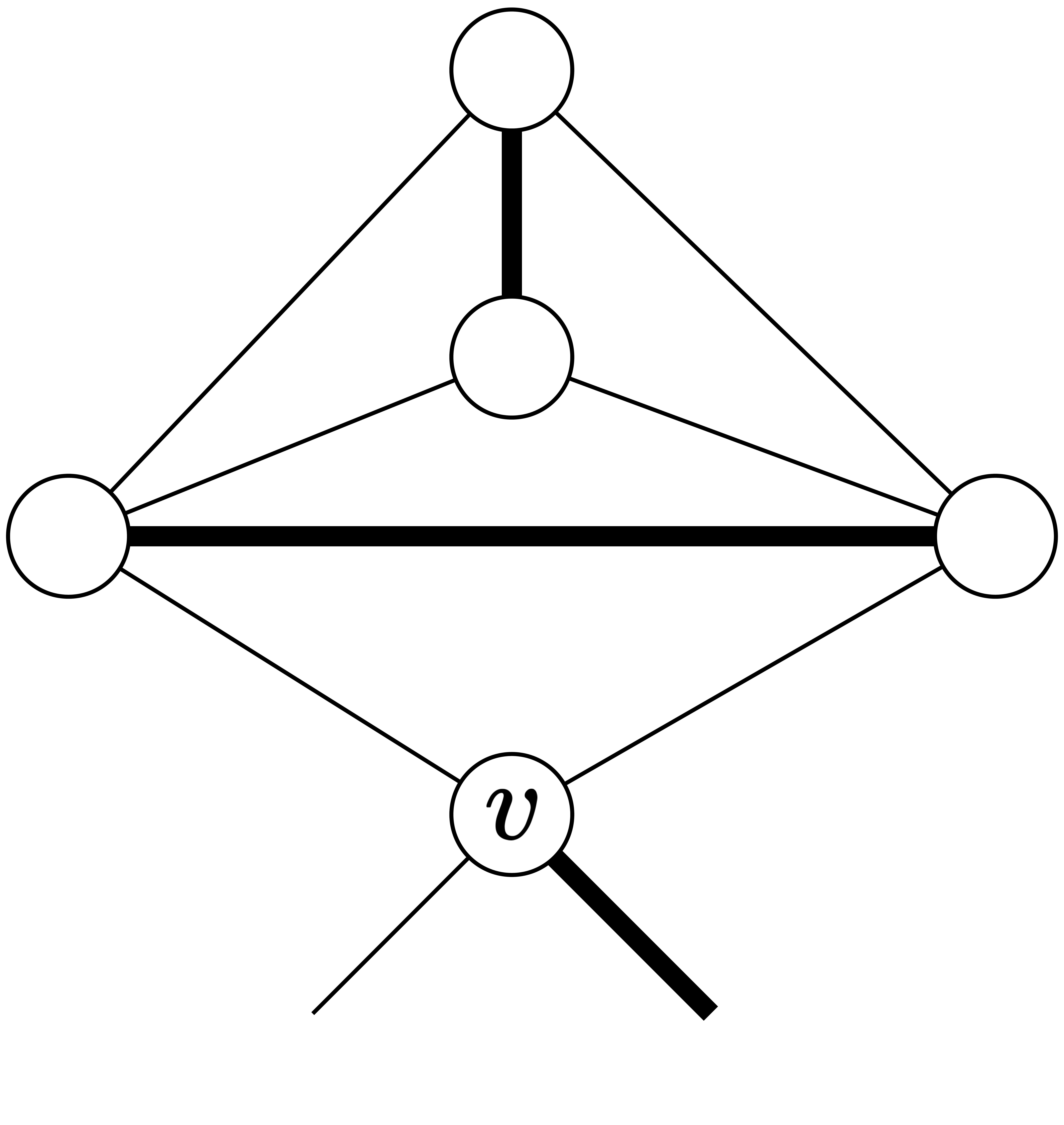}
   \end{center}
   \caption{How to transform $v$ a vertex of degree two in $G'$ by adding four vertices of degree three or four.}
   \label{deg2}
\end{figure}

To end the proof we show how we can transform $G'$ into $G''$ a planar graph with vertex degrees $d(v)\in\{3,4\}$. We note that in $G'$ the vertex degrees are $d(v)\in\{2,3,4\}$. To any vertex $v,\, d(v)=2,$ we connect four vertices as shown by Figure \ref{deg2}. Note that the graph induced by the four vertices has a perfect matching. So $G''$ is planar, $\delta(G'')=3$, and $\Delta(G'')=4$. Since the subgraph induced by $v$ and the four additional vertices is immune the result follows.
\end{proof}

We can observe in the proof of Theorem \ref{planarmax4} that the triangles are of great importance since they are immune. Hence, the question arises whether the Perfect Matching-Cut problem is still NP-complete for planar graphs with large girth.

For the Matching-Cut problem, Moshi \cite{Moshi} showed that every edge $uv$ could be replaced by a cycle of length four ($C_4=u-u_1-v-v_1-u$) so that it gives a bipartite graph instance that is equivalent to the original. Hence the Matching-Cut problem remains NP-complete for bipartite planar graphs with maximum degree eigth. Unfortunately it is clear that this construction does not work for the Perfect Matching-Cut problem and it doesn't seem feasible to achieve similar result with such construction. Bonsma \cite{Bonsma} showed that the Matching-Cut problem is NP-complete for planar graphs with girth five. For planar graphs with girth at least six, Bonsma et al. \cite{BonsmaFarley} showed that  all these graphs have a Matching-Cut. Beside the planar bipartite graphs, we were able to obtain a similar result by replacing each edge of a graph with the graph shown in Figure \ref{planargirth5}. We obtain the following.

\begin{theorem}
   Perfect Matching-Cut is NP-complete when restricted to planar graphs with girth five.
\end{theorem}
\begin{proof}
   From $G=(V,E)$ an instance of Perfect Matching-Cut where $G$ is planar we built a  planar graph $G'=(V',E')$ with girth five. Replace every edge $uv$ of $G$ by the gadget $H$ shown in Figure \ref{planargirth5_1}. Note that since the distance in $H$ between $u$ and $v$ is three then $g(G')=5$. From $M\subset E$ a perfect matching-cut of $G$ corresponds $M'\subset E'$ as follows. If $uv\in M$ then we take the bold edges shown in Figure \ref{planargirth5_2} in $M'$.  Clearly, this is the unique maximum matching when $u$ and $v$ are matched outside $H$. Moreover this matching does not disconnect $H$. Otherwise, when $uv\not\in M$, up to symmetry,  we take one of the two matchings represented by the bold edges in Figure \ref{planargirth5_3} and \ref{planargirth5_4} which both form a perfect matching-cut of $H$ that disconnect $u$ from $v$. One can check that there is no other perfect matching of $H$ when $u$ and $v$ are covered inside $H$.
 Also note that since $H$ has an even number of vertices there is no perfect matching-cut  where $u$ is matched inside $H$ and $v$ is matched outside $H$.
  Hence $M'$ is a perfect matching-cut of $G'$. Conversely, from the properties of  the perfect matchings of $H$ we described just below, if $G'$ has a perfect matching-cut $M'$ then there exists $M$ a perfect matching-cut of $G$.
\end{proof}

\begin{figure}[H]
   \centering
   \begin{subfigure}{.3\textwidth}
     \centering
     \includegraphics[width=1\linewidth]{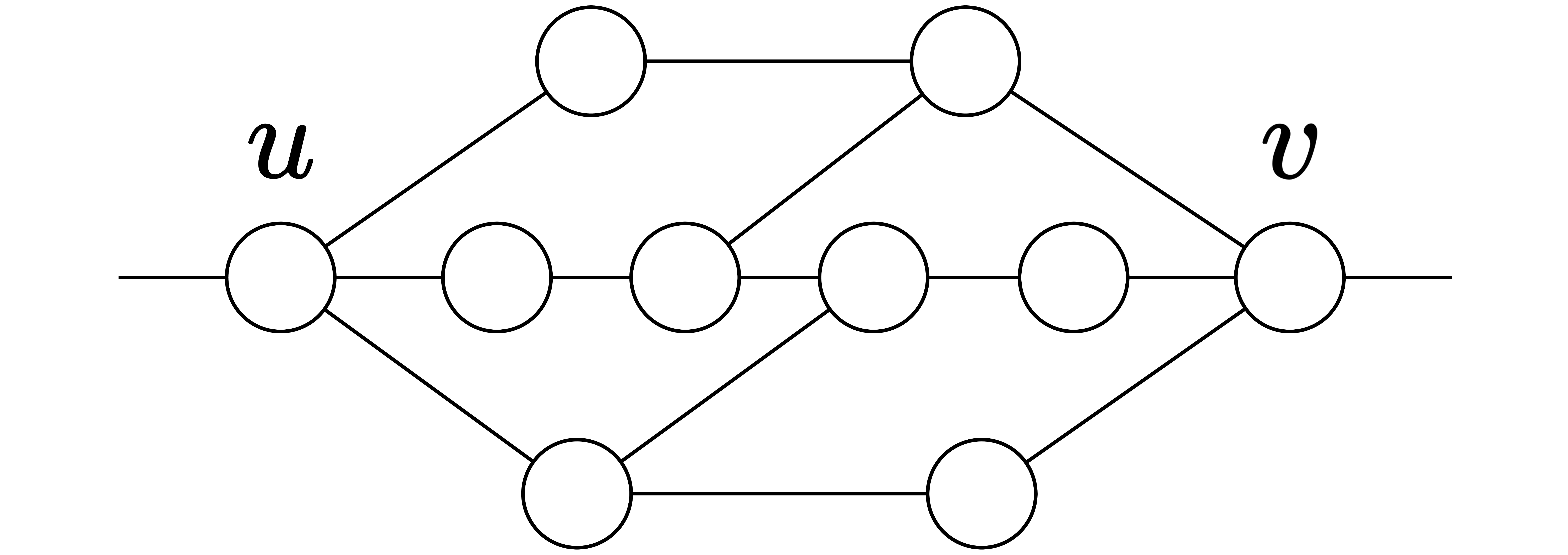}
      \caption{}
      \label{planargirth5_1}
   \end{subfigure}%
   \begin{subfigure}{.3\textwidth}
     \centering
     \includegraphics[width=1\linewidth]{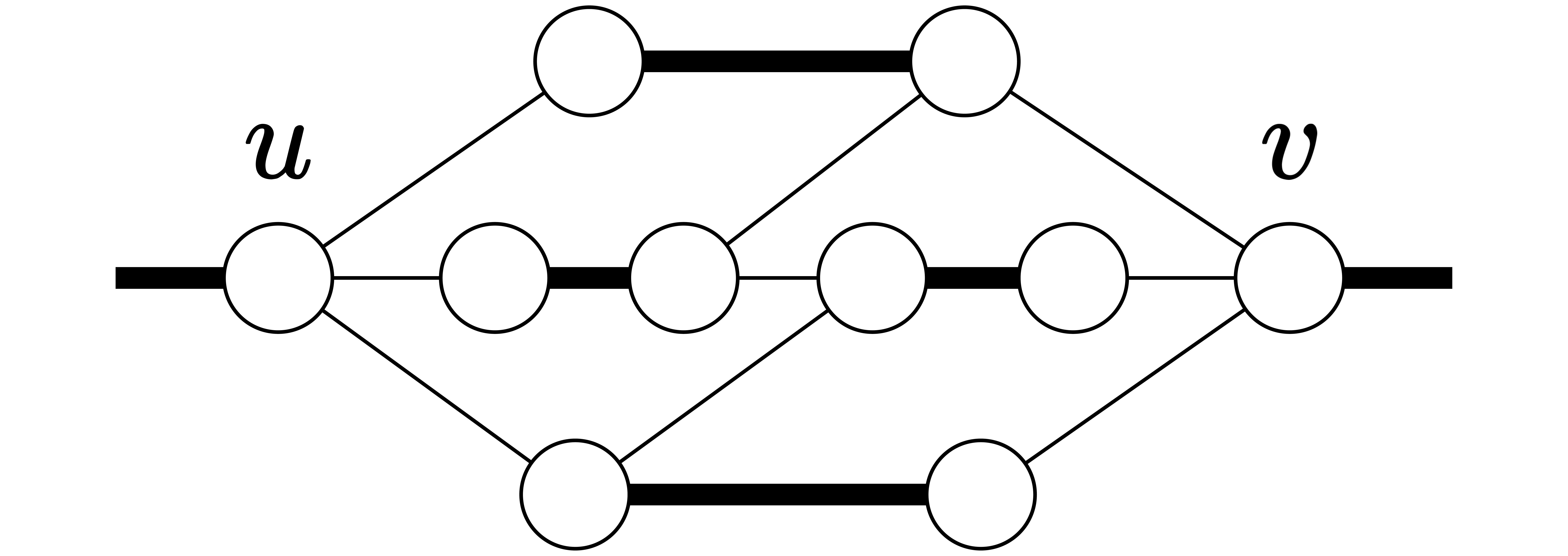}
      \caption{}
      \label{planargirth5_2}
   \end{subfigure}%

   \bigskip
   \begin{subfigure}{.3\textwidth}
      \centering
      \includegraphics[width=1\linewidth]{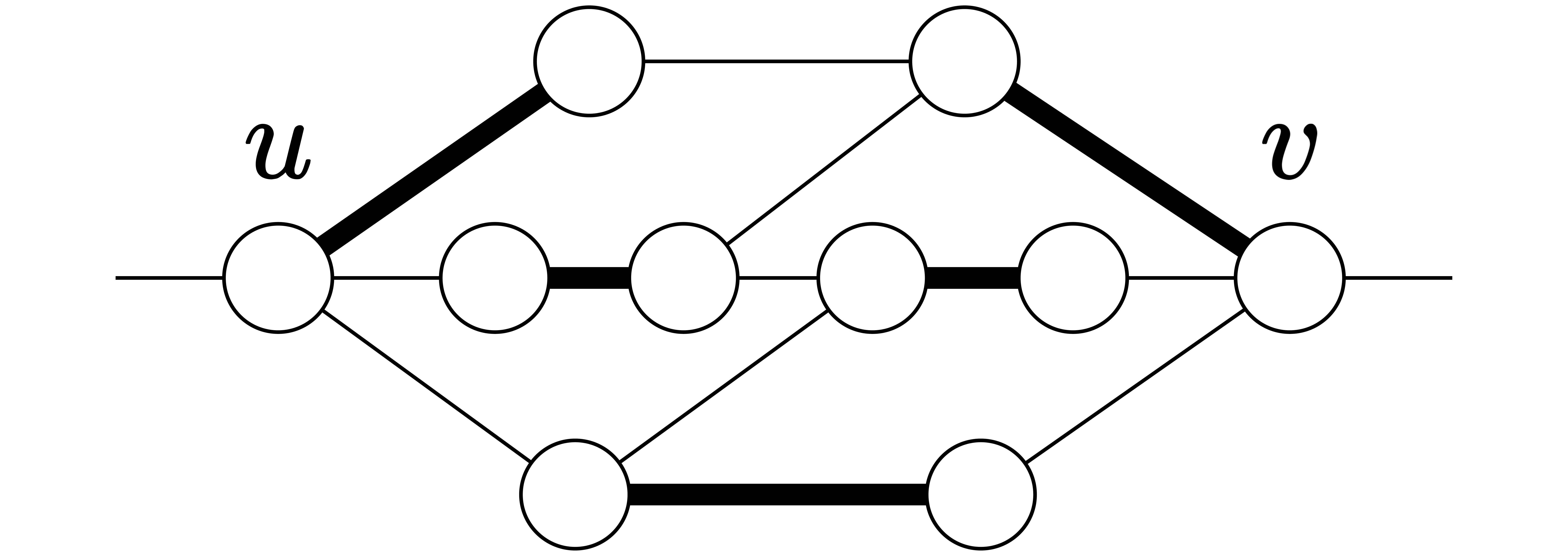}
      \caption{}
      \label{planargirth5_3}
   \end{subfigure}%
   \begin{subfigure}{.3\textwidth}
      \centering
      \includegraphics[width=1\linewidth]{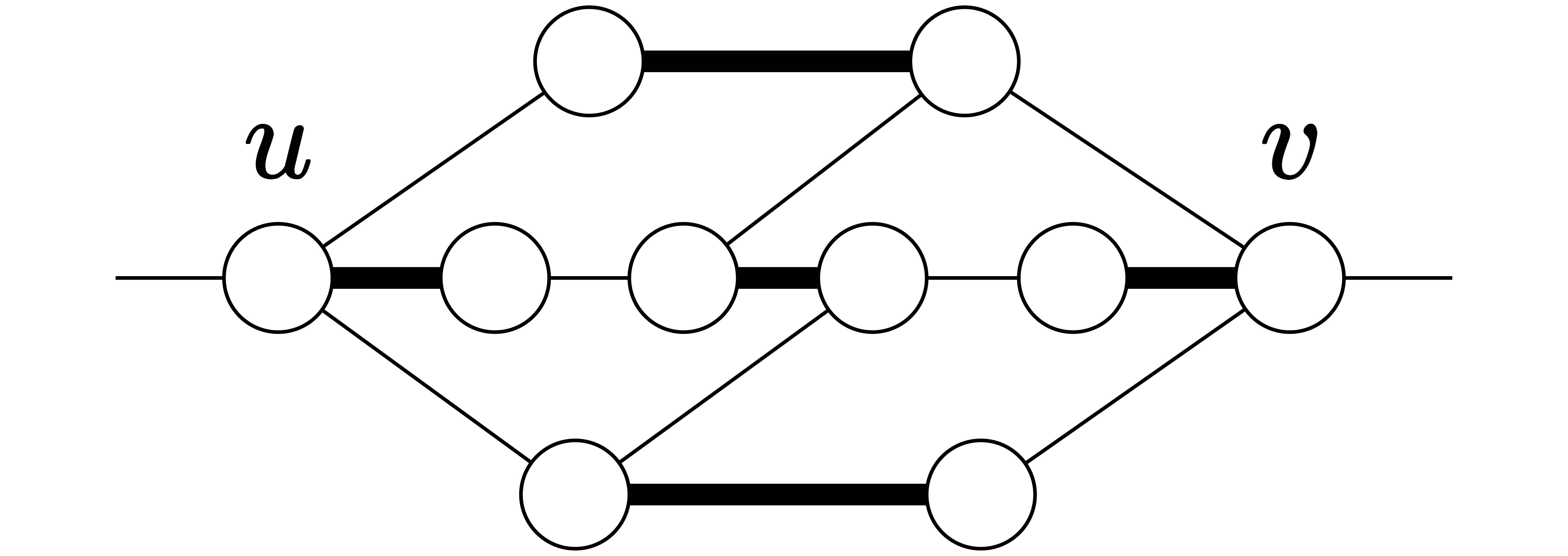}
      \caption{}
      \label{planargirth5_4}
   \end{subfigure}
   \caption{ The planar component $H$ with girth five (a).  The  perfect matching of $H\setminus \{u,v\}$ (b). The  two perfect matching-cuts of $H$ (c),(d).}
   \label{planargirth5}
\end{figure}

We use Theorem \ref{planarmax4} to prove the following.

\begin{theorem}
   Perfect Matching-Cut is NP-complete when restricted to planar $K_{1,4}$-free graphs with vertex degrees $\delta(v)\in \{3,4\}$.
\end{theorem}
\begin{proof}
   By Theorem \ref{planarmax4} Perfect Matching-Cut is NP-complete when restricted to planar graphs with vertex degrees $\delta(v)\in \{3,4\}$.
   Let $G=(V,E)$ be a planar graph with vertex degrees $\delta(v)\in \{3,4\}$ with $u\in V$ a vertex at the center of an induced $K_{1,4}$, that is, $G[\{u,w,x,y,z\}]=K_{1,4}$. Let $H$ be the graph where $u$ is replaced by the graph of five vertices $G_u$ has shown in Figure \ref{K14}. Note that $H$ is planar with vertex degrees $\delta(v)\in \{3,4\}$, and $H$ has one less $K_{1,4}$ as an induced subgraph than $G$. Also, $G_u$ is immune (every vertex is in a triangle) and $G_u\setminus \{u_1,u_2\}$ is odd. Therefore every perfect matching-cut of $H$ contains exactly one edge among $\{au_1,cu_1,bu_2,cu_2\}$ and exactly one edge among $\{u_1w,u_1z,u_2x,u_2y\}$.

   Let $M$ be a perfect matching-cut of $G$. W.l.o.g. $uw\in M$. Then $M\cup \{u_1w,ac,bu_2\}\setminus \{uw\}$ is a perfect matching-cut of $H$. Reciprocally, let $M'$ be a perfect matching-cut of $H$. W.l.o.g. $bc, u_1a, u_2x\in M'$. Since $G_u$ is immune, $M'\setminus \{bc, u_1a\}$ is a matching-cut of $H$. Therefore $M'\cup \{ux\} \setminus \{bc, u_1a, u_2x\}$ is a perfect matching-cut of $G$.

   Hence we can replace each $K_{1,4}$ from $G$ as done above, and the resulting graph has a perfect matching-cut if and only if $G$ has one.
\end{proof}

\begin{figure}[htbp]
   \begin{center}
   \includegraphics[width=15cm, height=2.5cm, keepaspectratio=true]{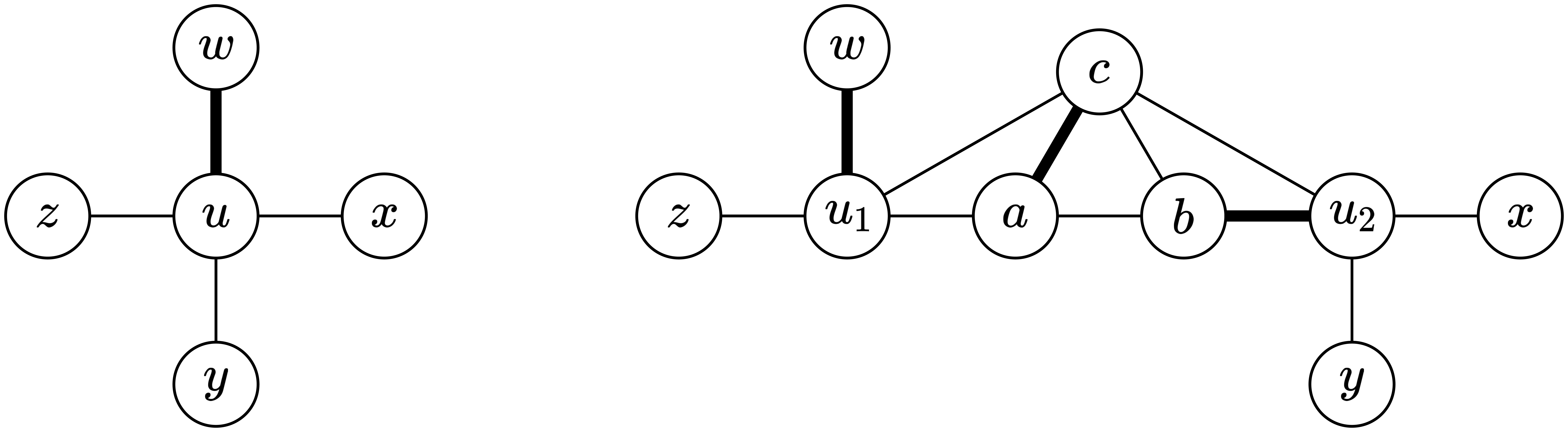}
   \end{center}
   \caption{How to replace a vertex $u$ at the center of a $K_{1,4}$ by $G_u$.}
   \label{K14}
\end{figure}

The remaining of this section deals with cubic graphs. For cubic planar graphs, we know the following.
\begin{theorem}[Diwan] \label{diwan}
   Every planar cubic bridgeless graph except $K_4$ has a perfect matching-cut.
\end{theorem}

We deal with the cubic graphs that have a bridge.
\begin{prop}
   Let $G=(V,E)$ be a cubic graph and $W\subset E$ be its set of bridges. Then for every perfect matching $M$ of $G$, we have $W\subset M$.
\end{prop}
\begin{proof}
   By contradiction, we suppose that there exists a bridge $uv$ such that $uv\notin M$. Let $C$ be one of the two connected component of $G-uv$. The restriction of $M$ to $C$ is a perfect matching of $C$. But $C$ is odd since it has an even number of vertices of degree $3$ and exactly one vertex of degree $2$, a contradiction.
\end{proof}

Hence every cubic graph that has a bridge and a perfect matching, also has a perfect matching-cut. So we have a complete overview for planar cubic graphs, that is.
\begin{coro}
   Every planar cubic graph except $K_4$ that has a perfect matching has a perfect matching-cut.
\end{coro}

Note that this does not hold for subcubic planar graph, see the graph on the left of Figure \ref{matchcut}. Also, it does not hold for cubic (non-planar) bridgeless graphs, since Diwan has shown in \cite{Diwan} that there exists an arbitrarily large class of cubic bridgeless graphs with a perfect matching that have no perfect matching-cut.

\section{$P_5$-free graphs}\label{P5free}
Here we prove that the Perfect Matching-Cut problem is polynomial for the class of $P_5$-free graphs.  We use the two theorems below.

\begin{theorem}[Bacs\'o and  Tuza \cite{Bacso}]
   Let $G$ be a connected $P_5$-free graph. Then $G$ has a dominating clique or a dominating induced $P_3$.
\end{theorem}
\begin{theorem}[Camby and Schaudt \cite{Camby}]
   Given a connected graph $G$ on $n$ vertices and $m$ edges, one can compute in time $O(n^5(n+m))$ a connected dominating set $X$ with the following property: for the minimum $k\geq 4$ such that $G$ is $P_k$-free, $G[X]$ is $P_{k-2}$-free or $G[X]$ is isomorphic to $P_{k-2}$.
\end{theorem}

This implies that for a conneted $P_5$-free graph, a connected dominating set $X$ is such that  $G[X]$ is $P_{3}$-free or $G[X]=P_3$ and can be computed in polynomial time. Note that a connected $P_3$-free graph is a clique. Hence it follows.

\begin{coro}\label{P5dom}
   Given a connected $P_5$-free graph $G$, computing a connected dominating set that is either a $P_3$ or a clique can be done in polynomial time.
\end{coro}

We prove the following.

\begin{theorem}
There is a polynomial time algorithm with the following specifications:

{\bf Input:} A connected $P_5$-free graph $G$.

{\bf Output:} Either $M$ a perfect matching-cut of $G$, or a determination that there is no  perfect matching-cut.

\end{theorem}
\begin{proof}
Let $G=(V,E)$ be a connected $P_5$-free graph. If $G$ has no perfect matching then there is no  perfect matching-cut.  If $G$ has a leaf then every perfect matching $M$ is a perfect matching-cut. Now $\delta(G)\ge 2$.

For $X\subseteq V$ the procedure $\cal P$ works as follows. Do the following as long as possible: if $v\in V\setminus X$ has two neighbors in $X$ then $X\leftarrow X\cup\{v\}$; if there exists $x\in X$ such that $x$ has two neighbors $u,v\in V\setminus X$ with $uv\in E$ then $X\leftarrow X\cup\{u,v\}$. In the following our objective is to try to construct a (perfect) matching-cut  with a cut between  $Z$ and $V\setminus Z$ such that $X\subseteq Z$.  Thus by Fact \ref{discclique} we apply $\cal P$ with $X$.

\begin{itemize}
\item A clique dominates $G$:

Let $K=\{k_1,\ldots,k_p\}$ be a dominating clique of $G$ and we assume that $p$ is minimum. Note that every vertex $v\not\in K$ has a neighbor $k_i\in K$.
\begin{itemize}
\item $p= 1$:  Every vertex is contained in a triangle and $k_1$ is in every triangle. By Fact \ref{discclique} there is no  perfect matching-cut.
\item $p= 2$: Let $V_1=\{v\in V\setminus K: vk_1\in E,\ vk_2\not\in E\}$, $V_2=\{v\in V\setminus K: vk_1\not\in E,\ vk_2 \in E\}$.

Firstly, we try to construct a perfect matching-cut from $X=\{k_1,k_2\}$. We apply $\cal P$ with $X$. Afterward, if $V=X$ then there is no  perfect matching-cut. When $V\ne X$, for every $v_1\in V_1\setminus X$, $v_1$ has all its neighbors in $V_2\setminus X$ and, for every $v_2\in V_2\setminus X$ then $v_2$ has all its neighbors in $V_1\setminus X$. Recall that $G$ has no leaves and therefore every vertex of $(V_1\cup V_2)\setminus X$ has at least one neighbor outside of $X$.

We show that $G[V\setminus X]$ is connected. For the sake of contradiction, let $u_1u_2$ and $v_1v_2$ be two edges in two distinct components of $G[V\setminus X]$, where $u_1,v_1\in V_1\setminus X$, $u_2,v_2\in V_2\setminus X$. Then $v_2-v_1-k_1-u_1-u_2=P_5$, a contradiction. If there exists $M$ a perfect matching-cut of $G$ then there exists a vertex of  $V\setminus X$, say $v_1\in V_1\setminus X$, such that $v_1$ is disconnected from $X$ in $G-M$,  so $v_1k_1\in M$. We show that $v_1$ has exactly one neighbor in $V_2\setminus X$. For contradiction let  $u_2,v_2$ be two neighbors of $v_1$ in $V_2\setminus X$. At least one edge among $u_2k_2,v_2k_2$ is an edge of $G-M$. Since $v_1v_2, v_1u_2\not\in M$ we have that $v_1$ is connected to $X$, a contradiction. Thus $v_1$ has exactly one neighbor $v_2\in V_2\setminus X$ and $v_2k_2\in M$. By symmetry, and since $G[V\setminus X]$ is connected,  $V\setminus X=\{v_1,v_2\}$ with $v_1v_2\in E$. Now let $M'$ be a maximum matching of $G[V\setminus\{k_1,k_2,v_1,v_2\}]$. If $M'$ is perfect then $M=M'\cup\{k_1v_1,k_2v_2\}$ is a perfect matching-cut of $G$ else there is no perfect matching-cut.

Secondly, we try to construct $M$ a perfect matching-cut such that $k_1$ is disconnected from $k_2$ in $G-M$. This means that we have to cut all edges between $X_1=N[k_1]\setminus \{k_2\}$ and $X_2=N[k_2]\setminus \{k_1\}$. Hence if $X_1\cap X_2\ne\emptyset$ or  $E(X_1,X_2)$ is not a matching-cut of $G$, then there is no perfect matching-cut. Otherwise, let $W$ be the set of vertices with an extremity in $E(X_1,X_2)$. Then let $M'$ be a maximum matching of $G-W$. If $M'$ is perfect then $M=E(X_1,X_2)\cup M'$ is a perfect matching-cut of $G$. Else there is no such perfect matching-cut.

\item $p\ge 3$: By Fact \ref{discclique} if a perfect matching-cut $M$ exists then all the vertices $k_1,\dots k_p$ are in a same component of $G-M$. We apply $\cal P$ with $X=K$.
If $V=X$ then there is no  perfect matching-cut. Now $V\ne X$. With the same argument as above $G[V\setminus X]$ is connected.
If there exists $M$ a perfect matching-cut then there must be a vertex $v$ of  $V\setminus X$ such that $v$ is disconnected from $X$ in $G-M$.  Let $vk_1\in M$. Let $u$ be a neighbor of $v$ in $\in V\setminus X$. We have  $uk_i\in M,i\ne 1$. It follows that if there exist two vertices of $V\setminus X$  with a common neighbor in $K$ then there is no  perfect matching-cut. Otherwise there exists $M'$ a (perfect) matching between the vertices of $V\setminus X$ and $K'\subset K$. Then let $M''$ be a maximum matching of $G[X\setminus K']$. If $M''$ if perfect then $M=M'\cup M''$ is a perfect matching-cut of $G$, else there is no such perfect matching-cut.
\end{itemize}

\item An induced $P_3$ dominates $G$:

Let $P_3=a-b-c$ be a dominating path of $G$. Note that every vertex $v\not\in\{a,b,c\}$ has a neighbor in $\{a,b,c\}$.
If $M$ is a perfect matching-cut then there is an edge-cut $E(X,V\setminus X)\subseteq M$ with at least two consecutive vertices of $P_3$ that are in $X$. By symmetry we have two cases to consider, that are, $\{a,b,c\}\subseteq X$ or $\{a,b\}\subseteq X$ and $c\in V\setminus X$.

\begin{itemize}

\item $\{a,b,c\}\subseteq X$: We apply $\cal P$ with $X=\{a,b,c\}$. If $V=X$ then there is no perfect matching-cut.
Now $V\setminus X\ne\emptyset$ and every  vertex $v\in V\setminus X$ has exactly one neighbor  in $P_3$, and a neighbor $u\in V\setminus X$. We show that $G[V\setminus X]$ is connected. For the sake of contradiction, we suppose that there are $v_1v_2$ and $u_1u_2$ two edges in two distinct components of $G[V\setminus X]$. Since the neighbors of $v_1,v_2$, respectively $u_1,u_2$, in $P_3$ are distinct we can suppose that $v_1,u_1$ have a common neighbor $w\in\{a,b,c\}$. Hence $v_2-v_1-w-u_1-u_2=P_5$, a contradiction. So $G[V\setminus X]$ is connected.

If there exists $M$ a perfect matching-cut then there exists $v\in V\setminus X$ such that $v$ is disconnected from $X$ in $G-M$. Let $w_v\in \{a,b,c\}$ be the neighbor of $v$ in $P_3$. We have $vw_v\in M$. Now for every neighbor $u$ of $v$ with $u\in V\setminus X$, let us denote $w_u$, $w_u\ne w_v,$ its neighbor in $P_3$.  Then $uw_u\in M$. Therefore $2\le \vert V\setminus X\vert\le 3$ and $E(X,V\setminus X)\subseteq M$ is a matching-cut. Let $M'$ be a maximum matching of $G[X\setminus N(V\setminus X)]$. If $M'$ is perfect then $M=M'\cup E(X,V\setminus X)$ is a perfect matching-cut of $G$, else there is no perfect matching-cut.

\item $\{a,b\}\subseteq X$ and $c\in V\setminus X$: We must find an edge-cut $M'$ so that there is no path from $c$ to $a,b$ in $G-M'$. Hence there are no paths  between $X_1=N[c]\setminus \{b\}$ and $X_2=N[a]\cup N[b]\setminus \{c\}$. If $X_1\cap X_2\ne\emptyset$ or  $E(X_1,X_2)$ is not a matching-cut then there is no such perfect matching-cut. Otherwithe, let $W$ be the set of vertices with an extremity in the matching-cut $E(X_1,X_2)$. Now let $M''$ be a maximum matching of $G - W$. If $M''$ is perfect then $M=E(X_1,X_2)\cup M''$ is a perfect matching-cut of $G$, else there is no such perfect matching-cut.
\end{itemize}
\end{itemize}
By Corollary \ref{P5dom} computing a $P_3$ dominating set or a clique dominating set is polynomial. Computing a maximum matching is polynomial.  Hence the algorithm is polynomial.\end{proof}

\begin{coro}
The Perfect Matching-Cut problem is polynomial for the classes of  cographs,  split graphs,  cobipartite graphs.
\end{coro}

\begin{proof}
These classes of graphs are subclasses of the class of $P_5$-free graphs.
\end{proof}

\section{Claw-free graphs}\label{claw}
We show that the Perfect Matching-Cut problem is polynomial for the class of claw-free graphs. For our proof
we use the theorem  proved by D. P. Sumner in \cite{Sumner}.
\begin{theorem}\label{Sumner}
Every connected claw-free graph with an even number of vertices has a perfect matching.
\end{theorem}
We can establish our result.
\begin{theorem}\label{clawfree}
Deciding if $G$ a connected claw-free graph has a perfect matching-cut is polynomial.
 \end{theorem}
\begin{proof}
By Fact \ref{leaf} we can assume that $G$ has a perfect matching and $\delta(G)\ge 2$.\\

Assume that $G$ has an induced path $P=a-b-c-d$ with $d_G(b)=d_G(c)=2$. Every perfect matching $M$ of $G$ is such that either
$M_P=\{ab,cd\}\subset M$ or $M_P'=\{bc\}\subset M$. Since $M_P$ is a matching-cut when $M_P\subset M$ we have that $M$ is a perfect matching-cut of $G$. Now let $M_P'=\{bc\}\subset M$. If $bc$ is a bridge then $M_P'$ is a matching-cut and $M$ is a perfect matching-cut. Otherwise, $G'=G-bc$ is connected and even, so by \cite{Sumner} $G'$ has a perfect matching $M'$. But we have  $M_P=\{ab,cd\}\subset M'$, so $M'$ is a perfect matching-cut of $G$. Hence, from now on, the graphs we are interested does not contain such a path $P$. \\

If $G$ is $C_3$-free then $G$ is an even cycle and there exists a perfect matching-cut. Now $G$ contains a triangle $C_3=\{a,b,c\}$. We show how to build $C$, an immune cluster, with $\{a,b,c\}\subseteq C\subseteq V$. We initialize $C=\{a,b,c\}$. We do the following two operations, in the following order, as long as possible. \\

{\it Rule 1:}  there exists $v\not\in C$ with two neighbors $s,t \in C$: then $C\leftarrow C\cup\{v\}$. This is mandatory since no matching can disconnect $v$ from $s$ and $t$.

{\it Rule 2:}  there exists  $v\in C$ with two neighbors $s,t\not \in C$: then $C\leftarrow C\cup\{s,t\}$. This is mandatory since: firstly, $v$ has a neighbor $w\in C$ which is not a neighbor of $s,t$ (by Rule 1); secondly,  $G$ is claw-free,  hence $\{v,s,t\}$ induces a triangle in $G$, and a triangle is immune.\\

Let $\mathcal{C}=\{C\}$. We apply the two previous operations for each triangle of $V$, that is not already in a cluster of $\mathcal{C}$, until every triangle belongs to a cluster. Since $G$ is claw-free and $\delta(G)\ge 2$, every vertex that is not in a cluster has a degree two. Now $\mathcal{C}$ is the set of clusters of $G$.

We say that two distinct clusters $C,C'$ are {\it linked} if there exists a path $P=u-\cdots -v$ with $u\in C,\, v\in C'$, such that every vertex $w$ of $P$, $w\ne u,v$, is not contained in a cluster. Alternatively, we say that $P$ is a {\it link} between $C$ and  $C'$. Notice that there might exist several links between two clusters. Since $G$ is claw-free, two links cannot have a same extremity. Hence, when there exists an edge between two links, the two endpoints of this edge are in a same cluster.

Let $C$ be a cluster. Its {\it core}  $K\subseteq C$ consists of all the vertices that are not an extremity of a link, that is, $N[K]\subseteq C$. Note that $K$ is connected. The {\it corona} of $C$ is $Q=C\setminus K$.  Hence  every  $v\in Q$ is the extremity of one link. We say that $C$ is \textit{even} when its core $K$ is such that $\vert K\vert$ is even, otherwise it is {\it odd}. Let $H=({\cal C},{\cal E})$ be the graph where ${\cal C}$ is the current set of clusters, and $C_iC_j\in {\cal E}$ when there exits a link between $C_i$ and $C_j$. \\

We consider the following situations numbered by the order they are  performed in our algorithm: \\

\textit{Case 1}: there exists only one cluster $C\in\cal C$, that is $C=V$.\\

Since $C$ is immune then $G$ has no perfect matching-cut. \\

\textit{Case 2}: there exists an even cluster $C$ such that $G[V\setminus C]$ is connected. \\

By Theorem \ref{Sumner}, $G[K]$ has a perfect matching $M_K$, $K$ being the core of $C$. By \textit{Rule 1} and \textit{Rule 2}, $M_Q=E(Q,N(Q)\setminus C)$ is a perfect matching between each vertex of the corona $Q$ of $C$ and its unique neighbor outside $C$. Note that $M_Q$ is a matching-cut that separates $C$ from the rest of the graph. We show that $G'=G[V\setminus N[C]]$ is also connected. By contradiction we assume there exist two vertices $s,t\in V\setminus N[C]$ that are not connected in $G'$. In $G[V\setminus C]$ there exists $P$ a path between $s$ and $t$. By \textit{Rule 1} we have $P=s-\cdots-u-v-w-\cdots-t$ with $v\in N[C]\setminus C$ and $u,w\in V\setminus N[C]$. But $G$ is claw-free, so $uw\in E$ and $P'=s-\cdots-u-w-\cdots-t$ is a path of $G'$, a contradiction.
 Hence when $G[V\setminus C]$ is connected we have that $G[V\setminus N[C]]$ is also connected. Therefore let $M_G$ be a perfect matching of $G'$. Then $M_G\cup M_Q$ is a perfect matching-cut of $G$. \\

\textit{Case 3}: there exists an odd cluster $C$ such that $G[V\setminus C]$ is connected, and there is $v\in N(C)$ that does not belong to any cluster. \\

By \textit{Rule 1}, $v$ has only one neighbor $w\in Q$. Recall that $v$ is the unique neighbor of $w$ outside $C$ and $d_G(v)=2$.  We replace $C$ by $C'=C\cup \{v\}$ in $\cal C$. Hence $K'=K\cup \{w\}$ is the even core of $C'$, so $C'$ is an even cluster.
Note that $G[V\setminus C']$ is connected and non empty. Therefore we are as  in the \textit{Case 2} and $G$ has a perfect matching-cut.\\

\textit{Case 4}: $H$ has a leaf $C$. \\

Let $C'$ be the unique neighbor of $C$ in $H$. Since the core $K$ of $C$ is odd, $G[K]$ has no perfect matching. It follows that for every perfect matching $M$ of $G$, at least one link between $C$ and $C'$ is not in $M$. So $C$ and $C'$ cannot be disconnected by a perfect matching-cut. Hence $C$ and $C'$ are merged into a new cluster $\bar C=C\cup C'$. Thereby, the total number of clusters in $\cal C$ decreases by one unit, and we return to \textit{Case 1}. \\

\textit{Case 5}: there is a pair of odd clusters $C,C'$ such that $CC'\in \cal E$ and $G[V\setminus (C\cup C')]$ is connected. \\

Let $\bar C=C\cup C'$. Recall that after \textit{Case 3} is performed the links between $C$ and $C'$ are edges. Note that $E(C,C')$ is a matching. Hence $\bar C$ is even since its core is $\bar K=K\cup K'\cup W$, where $W$ is the set of vertices with an endpoint in $E(C,C')$. Since we are after \textit{Case 4}, $N(\bar C)\setminus \bar C\neq \emptyset$. Then we replace $C$ and $C'$ by $\bar C$ in $\cal C$ and we are as  in the \textit{Case 2} and $G$ has a perfect matching-cut. \\

In order to show that all the situations are covered by \textit{Case 1} to \textit{Case 5} we need the following two facts. First, by \cite{Bondy} (p.211, 9.1.6) we have that for every $2$-connected graph $G'=(V',E')$, there exists a {\it contractible} edge, that is an edge $e$ such that $G'/e$ is still a $2$-connected graph. Second, we show that for every $1$-connected graph $G'=(V',E')$ with $\delta(G')\ge 2$, there exists $e$ an edge of a terminal component  such that $G'/e-e$ is still a connected graph.

Let $T$ be a terminal component with $a$ its vertex-cut. Since $G'$ has no leaf then $T$ has a cycle containing $a$. Let $D=a-\cdots-u-v-\cdots-a$ be a longest cycle and let $e=uv$ be an edge of $D$ with $u,v\ne a$. For contradiction we assume that $G''=G'/e-e$ is not connected. Then $G''$ has a connected component $T'$ that does not contain $a$ nor the other vertices of $D$. Let $w$ be a vertex of $T'$. Since $T$ is biconnected then in $G'$ there exists a cycle $D'=a-\cdots-u-\cdots-w-\cdots-v-\cdots-a$, contradicting the maximality of $D$.\\

We show that the algorithm terminates.
Recall that after \textit{Case 4}, the current graph $H$ is connected and $\delta(H)\geq 2$. Also, after \textit{Case 5} all the clusters are odd. By the first fact above when $H$ is $2$-connected we can always perform \textit{Case 2} or \textit{Case 3} or \textit{Case 5}. Otherwise $H$ is $1$-connected and by the second fact again we can perform \textit{Case 2} or \textit{Case 3} or \textit{Case 5}.

When \textit{Case 1} is performed the algorithm stops and $G$ has no perfect matching-cut. When \textit{Case 2} or \textit{Case 3} is performed the algorithm stops and $G$ has a perfect matching-cut. When \textit{Case 4} or \textit{Case 5} is performed the number of cluster in $\cal C$ decreases by one unit and we go back to \textit{Case 1}.\\

It remains to show that the algorithm is polynomial. Recall that Edmonds's Algorithm computes a perfect matching  in $O(\vert E\vert.\vert V\vert^2)$ (better algorithms are known). Checking if there exists an induced path $P=a-b-c-d$, with $d_G(b)=d_G(c)=2$ can be done in $O(\vert V\vert^2)$.

At each step all the clusters are vertex-disjoint so there are at most $O(\vert V\vert)$ clusters. Each cluster contains at least a triangle, so initializing the clusters take at most $O(\vert V\vert^3)$. Applying {\it Rule 1} and {\it Rule 2} can be done in $O(\vert V\vert^3)$.  Thus the initialization takes $O(\vert V\vert^3)$.

Searching for two adjacent odd clusters in a same terminal biconnected component takes $O(\vert E\vert)$. \textit{Case 5} is performed at most $\vert V\vert$ times. Thus all the applications of  \textit{Case 1} and \textit{Case 5} can be done in $O(\vert V\vert^3)$. \textit{Case 2} or \textit{Case 3} is performed at most one time. These two cases need to determine a perfect matching.
Hence the  overall complexity of our algorithm is $O(\vert E\vert.\vert V\vert^2)$.
\end{proof}

\section{Graphs with fixed Bounded Tree-Width}\label{width}

It is shown in \cite{Bonsma} that the graph property of having a Matching-Cut can be expressed in MSOL. All graph properties definable in MSOL can be decided in linear time for the classes of graphs with bounded tree-width, when a tree-decomposition is given. Hence it can be decided in polynomial time if a graph of bounded tree-width (given a tree-decomposition) has a Matching-Cut. We refer to \cite{Courcelle} for definitions and an overview of the logical language MSOL.

\begin{theorem}
   Let $G=(V,E)$ be a graph of bounded tree-width. Deciding if $G$ has a Perfect Matching-Cut can be done in polynomial time.
\end{theorem}
\begin{proof}
  We have adapted the MSOL formulation of the Matching-Cut property shown in \cite{Bonsma} so that it corresponds to the perfect matching-cut. The property of having a perfect matching-cut in the graph $G$ can be expressed in MSOL as follows.

   \begin{gather*}
      \exists V_1\subseteq V: \exists V_2\subseteq V:\exists M\subseteq E: (V_1\cap V_2=\emptyset)\land (V_1\cup V_2=V)\land \neg (V_1=\emptyset)\land \neg (V_2=\emptyset)\land \\
      \neg (\exists u\in V_1:\exists v\in V_2:\exists w\in V_2:\neg (v=w)\land (uv\in E)\land (uw\in E))\land \\
      \neg (\exists u\in V_2:\exists v\in V_1:\exists w\in V_1:\neg (v=w)\land (uv\in E)\land (uw\in E))\land \\
      (\forall u\in V:\exists v\in V:\forall w\in V: \neg (v=w)\land (uv\in M)\land \neg (uw\in M)).
   \end{gather*}

   Note that the first three lines correspond to the matching-cut property, that is, no vertex of $V_i$ has more than one neighbor in $V_j$, $i\neq j$. The last line is added so that each vertex of $V$ has exactly one endpoint in the matching $M$, whether it is in $V_1$ or $V_2$. Hence a vertex $v\in V_i$ with no neighbor in $V_j$, $i\neq j$, must have exactly one neighbor $u\in V_i$ such that $uv\in M$.
\end{proof}

It is also shown in \cite{Bonsma} that the graph property of having a Matching-Cut can be expressed in MSOL without \textit{quantification over edge sets}. Any graph property expressible as MSOL without \textit{quantification over edge sets} can be decided in linear time for classes of graphs with bounded clique-width (such as co-graphs), when a corresponding decomposition is given. We refer to \cite{Courcelle} for additional details.
Unfortunately, it has been proved in \cite{Courcelle},  Proposition 5.13 page 338, that the Perfect Matching cannot be expressed in MSOL without quantification over edge sets. Hence we cannot conclude that it can be decided in polynomial time if a graph of bounded clique-width (given a corresponding decomposition) has a Perfect Matching-Cut using the associated MSOL definition.

\section{Conclusion and open problems}\label{conc}
With a same flavour as for the Matching-Cut problem we proved complexity results for the Perfect Matching-Cut problem under several  parameter restrictions or for graph subclasses.

\begin{itemize}
\item Regular graphs: PMC is NP-complete for $5$-regular graphs even for bipartite graphs;
\item Diameter: For $d$ a fixed interger and $G$ a graph with $diam(G)=d$, then PMC is polynomial for $d\le 2$ and NP-complete for $d\ge 3$; when $G$ is bipartite PMC is polynomial for $d\le 3$ and NP-complete for $d\ge 4$;
\item Planar graphs: PMC is NP-complete for graphs with  $\delta(G)=3,\Delta(G)=4$, and for graphs with girth $g(G)=5$, but is polynomial for cubic planar graphs;
\item PMC is polynomial for claw-free graphs and NP-complete for $K_{1,4}$-free planar graphs;
\item PMC is polynomial for $P_5$-free graphs;
\item Bounded treewidth: PMC is polynomial.
\end{itemize}
We give a list of open problems  that seems relevant after the results we proved above.
\begin{itemize}
\item Cubic (nonplanar) graphs, subcubic graphs, $4$-regular planar graphs;
\item Bipartite planar graphs;
\item Planar graphs with girth $g(G)=d$ for fixed $d\ge 6$;
\item $P_k$-free graphs for $k\ge 6$.
\end{itemize}

 \begin{ack}
The authors  express their gratitude to Fran\c cois Delbot and St\'ephane Rovedakis for helpful discussions.
\end{ack}

\end{document}